\documentclass{article}
\usepackage{amsmath}
\usepackage{amsthm}
\usepackage{amsmath,bm}
\usepackage{amssymb}
\usepackage{color}
\usepackage{xcolor}
\usepackage{lineno}

\usepackage[inner=2cm,outer=1.5cm]{geometry}
\usepackage[colorlinks=true,linkcolor=blue]{hyperref}

\newtheorem{theorem}{Theorem}[section]

\newtheorem{theorem2}[theorem]{Remark}
\newtheorem{theorem3}[theorem]{Lemma}

\newtheorem{corollary}[theorem]{Corollary}

\def\g{\gamma}
\def\eps{\bm \epsilon}
\def\sig{\bm \sigma}

\def\ba{\bm \alpha}
\def\bv_0{\bm \psi}
\def\bU{\boldsymbol U}
\def\bUd{\tilde{\boldsymbol U}}
\def\bV{\boldsymbol V}
\def\bVd{\tilde{\boldsymbol V}}
\def\bP{\boldsymbol P}
\def\bW{\boldsymbol W}
\def\bP{\boldsymbol P}
\def\bu{\boldsymbol u}
\def\bv{\boldsymbol v}
\def\bp{\boldsymbol p}
\def\bw{\boldsymbol w}
\def\bz{\boldsymbol z}
\def\bq{\boldsymbol q}
\def\bn{\boldsymbol n}
\def\bx{\boldsymbol x}

\def\bf{\boldsymbol f}

\def\bg{\boldsymbol g}

\def\bbf{\boldsymbol b}
\def\bK{\boldsymbol K}

\def\bM{\mathcal M}
\def\bL{\mathcal L}
\def\bD{\mathcal D}
\def\bF{\mathcal F}
\def\bG{\mathcal G}
\def\bH{\mathcal H}
\def\bA{\mathcal A}
\def\br{\boldsymbol r}
\def\bs{\boldsymbol s}
\def\by{\boldsymbol y}
\def\divv{\text{div}}

\def\Divv{\text{Div}}
\def\Divu{\text{\underline{Div}}}

\newcommand{\msp}{\hspace{-0.2em}}
\def\bud{\dot{\boldsymbol u}}
\def\bvd{\dot{\boldsymbol v}}
\def\bwd{\tilde{\boldsymbol w}}
\def\bzd{\tilde{\boldsymbol z}}
\providecommand{\keywords}[1]
{
	\small	
	\textbf{\textit{Keywords---}} #1
}

\begin{document}

\title{On the stability of time-discrete dynamic multiple network poroelasticity systems
arising from second-order implicit time-stepping schemes}
\author{Fadi Philo}
\maketitle

\begin{abstract}
The classical Biot's theory provides the foundation of a fully dynamic poroelasticity model describing the
propagation of elastic waves in fluid-saturated media. Multiple network poroelastic theory (MPET) takes
into account that the elastic matrix (solid) can be permeated by one or several ($n\ge1$) superimposed
interacting single fluid networks of possibly different characteristics; hence the single network (classical Biot)
model can be considered as a special case of the MPET model.

We analyze the stability properties of the time-discrete systems arising from second-order implicit time stepping
schemes applied to the variational formulation of the MPET model and prove an inf-sup condition with a constant
that is independent of all model parameters. Moreover, we show that the fully discrete models obtained for a family
of strongly conservative space discretizations are also uniformly stable with respect to the spatial discretization
parameter. The norms in which these results hold are the basis for parameter-robust preconditioners.
\end{abstract}


\section{Fully dynamic poroelasticity models: The continuous case}
In this section we formulate the continuous dynamic models whose stable and mass-conservative discretization we will address in the further course of this work.
\subsection{The dynamic Biot model}
Let us start with the single network model, which we will also refer to as dynamic Biot problem. For an open domain
$\Omega \subset \mathbf{R}^d, d= 2,3$, the unknown physical variables in the dynamic Biot problem we are going
to consider are the displacement $\bu$ of the solid matrix occupying~$\Omega$,
the relative displacement $\bw:=\varphi(\bv-\bu)$ of the fluid, denoting by $\bv$ the displacement of the fluid and by
$\varphi \in (0,1)$ the porosity of the solid\footnote{According to~\cite{Biot1955theory} the porosity $\varphi$ is defined
as $\varphi = \frac{V_p}{V_b}$ where $V_p$ is the volume of the pores contained in a sample of bulk volume $V_b$.},
and the fluid pressure $p$, cf.~\cite{Biot1955theory}.

In the regime of linear elasticity (assuming Hook's law) we have the relations
\begin{subequations}
\begin{align}
&\sig(\bu)= 2\mu\eps(\bu)+ \lambda \divv(\bu) \textbf{I} ,\label{linE1} \\
&\eps(\bu)= \frac{1}{2}\left(\nabla\bu+(\nabla \bu)^T\right) \label{linE2}
\end{align}
\end{subequations}
between the total stress $\sig=\sig(\bu)$, the strain $\eps=\eps(\bu)$ and the displacement field $\bu$. Defining
the total density $\rho$ of the fluid-saturated porous medium by
\begin{equation}\label{rho_tot}
\rho:= \varphi \rho_f + (1-\varphi) \rho_s
\end{equation}
in terms of the fluid density $\rho_f$ and the solid density $\rho_s$, the first equation of motion reads
\begin{equation}\label{eqM1}
-\divv\sig + \rho \ddot{\bu}  + \rho_f \ddot{\bw} +\tilde{\alpha} \nabla p =  \tilde{\bf}, \quad \text{in } \Omega\times(0,T)
\end{equation}
where $\tilde{\alpha}\in[\varphi,1]$ denotes the Biot-Willis parameter\footnote{For physical reasons it is natural to assume
that $\varphi \leq \tilde{\alpha} \leq 1$, cf.~\cite{Biot1957elastic}.} and $\tilde{\bf}$ the body force density, cf.~\cite{Biot1957elastic}.

The second equation of motion, describing the momentum balance of the fluid component, is given by
\begin{equation}\label{eqM2}
\rho_f \ddot{\bu}+\rho_{m_f} \ddot{\bw} + \bK^{-1} \dot{\bw} +\nabla p = - \tilde{{\bf}}_f, \quad \text{ in } \Omega\times(0,T)
\end{equation}
where $\rho_{m_f} := \rho_f/\varphi$ is the effective fluid density, $\bK:=\bm{\kappa}/\eta$ the hydraulic conductivity of
the medium for a fluid with viscosity $\eta$, and $\bm{\kappa}$ the permeability tensor, which for simplicity here will be
assumed to be of the form $\bm{\kappa} = \kappa \bm{I}$ for a scalar permeability coefficient $\kappa$.\footnote{The
latter assumption is valid for isotropic porous media, cf.~\cite{Biot1941general}.}
If the only body forces are due to gravity then the total body force and fluid body force are given by $\tilde{\bf} = \rho \bbf$
and $\tilde{\bf}_f =\rho_f \bbf$ with $\bbf$ denoting the gravitational acceleration.

The system is closed by the mass conservation equation
\begin{equation}\label{eqM3}
-\tilde{\alpha}\dot{\divv \bu} -\dot{\divv \bw} -c_{p} \dot{p} = 0, \quad \text{in } \Omega\times(0,T)
\end{equation}
where $c_p$ is the constrained specific storage coefficient.
Note that if both the elastic solid and the fluid are incompressible one has $c_p=0$, cf~\cite{Biot1962}, a situation which is also covered by
the analysis presented in section~\ref{sec:stability} which provides stability in this case as well.

For symmetry reasons, it is convenient to transform the equations~\eqref{eqM1}--\eqref{eqM3} into an equivalent
system which after variational formulation produces a saddle point problem in each step of an implicit time integration
method. This can be achieved by first inserting the right hand side of the definition $\bw:=\varphi(\bv-\bu)$ of $\bw$
in~\eqref{eqM1}--\eqref{eqM3}, then multiplying equation~\eqref{eqM2} with $-\varphi$ and adding it to
equation~\eqref{eqM1} to obtain a new equation replacing~\eqref{eqM1}. Finally, substituting $\varphi \bv$ with $\bv$
and denoting $0 \le \alpha := \tilde{\alpha} - \varphi$ one ends up with the new system
\begin{subequations}\label{eq:Biot}
\begin{align}
-\divv\sig + ((1-\varphi)\rho_s - \varphi \rho_f+\varphi^2 \rho_{m_f}) \ddot{\bu} 
+ \varphi^2 \bK^{-1} \dot{\bu}
+ (\rho_f  -\varphi \rho_{m_f}) \ddot{\bv}  - \varphi \bK^{-1} \dot{\bv}
+\alpha \nabla p &= \bf, \label{eq:Biot_a} \\
(\rho_f  - \varphi \rho_{m_f}) \ddot{\bu} - \varphi \bK^{-1}\dot{\bu}+\rho_{m_f} \ddot{\bv} + \bK^{-1}\dot{\bv}+\nabla p 
&= \bg,  \qquad  \label{eq:Biot_b}\\
-\alpha \dot{\divv \bu}  -\dot{\divv \bv}  -c_{p}\dot{p} &= 0. \label{eq:Biot_c}
\end{align}
\end{subequations}
where we have also used the notation $\bf:=\tilde{\bf} + \varphi \tilde{\bf}_f$ and $\bg:=-\tilde{\bf}_f$.

The dynamic Biot problem \eqref{eq:Biot} has to be complemented by proper initial conditions at time $t=t_0$,
e.g., prescribing $\bu(\bx,0)=\bu^{(0)}(\bx)$, $\bv(\bx,0)=\bv^{(0)}(\bx)$, $p(\bx,0)=p^{(0)}(\bx)$,
$\dot{\bu}(\bx,0)=\bu^{(1)}(\bx)$, $\dot{\bv}(\bx,0)=\bv^{(1)}(\bx)$ at time $t_0=0$
as well as proper boundary conditions at any time $t>t_0$, e.g.,
\begin{subequations}\label{eq:MPET_BC}
	\begin{eqnarray}
	p(\bx,t) &=& p_{D}(\bx,t)  \quad \mbox{for } \bx \in \Gamma_{p,D}, \quad t > 0,\\ 
	\bK \frac{\partial p (\bx)}{\partial \bn} &=& q_{N}(\bx,t)  \quad \mbox{for }  \bx \in \Gamma_{p,N}, \quad t > 0,\\ 
	\bu(\bx,t) &=& {\bu}_D(\bx,t)  \quad \mbox{for }  \bx \in \Gamma_{\bu,D}, \quad t > 0, \\ 
	({\sig(\bx,t)}-\alpha p  \bm I) \, {\bn} (\bx) &=& {\bm g}_N(\bx,t) \quad \mbox{for }  \bx \in \Gamma_{\bu,N}, \quad t > 0,
	\end{eqnarray}
\end{subequations}
where
$\Gamma_{p,D} \cap \Gamma_{p,N} = \emptyset$,
$\overline{\Gamma}_{p,D}\cup \overline{\Gamma}_{p,N}=\Gamma=\partial{\Omega}$
and
$\Gamma_{\bu,D} \cap \Gamma_{\bu,N} = \emptyset$,
$\overline{\Gamma}_{\bu,D} \cup \overline{\Gamma}_{\bu,N}=\Gamma$. 
A more detailed derivation of the system~\eqref{eq:Biot} and some fundamental results regarding its well-posedness
can be found in~\cite{Dafermos1968existence,Carlson1972linear,Zienkiewicz1982basic,Showalter2000diffusion}.

In compact notation, the system~\eqref{eq:Biot_a}--\eqref{eq:Biot_c} can be written in the form
\begin{equation}\label{eq:dyn_sys}
\bM \, \ddot{\by} + \bD \dot{\by} + \bL \by = \bF
\end{equation}
with operators $\bM$, $\bD$, $\bL$, right hand side $\bF$, and unknown vector $\by$ given by
\begin{subequations}
\begin{align}
\bM&=
\begin{bmatrix}
((1-\varphi) \rho_s - \varphi \rho_f+ \varphi^2 \rho_{m_f}) I & (\rho_f  -\varphi \rho_{m_f}  ) I & 0 \\
(\rho_f  - \varphi \rho_{m_f}) I & \rho_{m_f} I & 0 \\
0 & 0 & 0
\end{bmatrix}, \quad \label{eq:M} \\
\bD&=
\begin{bmatrix}
\varphi^2 \bK^{-1}  & - \varphi \bK^{-1} & 0 \\
- \varphi \bK^{-1}  & \bK^{-1} & 0\\
-\alpha \divv & - \divv & - c_p I
\end{bmatrix},  \quad \label{eq:D}  \\
\bL&=
\begin{bmatrix}
- 2\mu\divv\eps -\lambda \nabla\divv  & 0 & \alpha \nabla \\
0  & 0& \nabla\\
0 & 0& 0
\end{bmatrix},  \quad \label{eq:L} 
\end{align}
\end{subequations}
and
\begin{equation}\label{eq:Fr1}
\bF=
\begin{bmatrix}
\bf \\ \bg \\ 0
\end{bmatrix},  \quad
\by=
\begin{bmatrix}
\bu \\ \bv \\ p
\end{bmatrix}, 
\end{equation}
Many problems in structural dynamics can be represented in the abstract form~\eqref{eq:dyn_sys}.
Note that in case of the dynamic Biot model $\bM+\bL+\bD$ is a self-adjoint and invertible linear
operator with an inverse $(\bM+\bL+\bD)^{-1}$ defined on the dual space
$\bW^*:=\bUd^* \times \bU^* \times P^*$ of an appropriate product space $ \bW:=\bUd \times \bU \times P$.
However, the operators $\bM: \bW \to \bW^*$ and $\bL: \bW \to \bW^*$ and in the case $c_p=0$
also $\bD: \bW \to \bW^*$ are not invertible individually.
This is the reason why many popular standard implicit time integration schemes, for instance the
 Crank Nicolson method, cf~\cite{Hairer_etal2010geometric}, which requires the invertibility of $\bM$, can not be
applied straighforwardly. However, further refined/combined methods have already been considered in
the present context as early as in \cite{Zienkiewicz1984dynamic}. Before we will also address this issue
we will generalize the system~\eqref{eq:Biot} in order to present the dynamic MPET model as subject
of further discussions.

\subsection{The dynamic MPET model}

A basic assumption in the MPET model is that the elastic solid matrix is permeated by $n \ge 1$
fluid networks each of which being described by its individual fluid displacement $\bv_i$, relative
fluid displacement $\bw_i=\varphi_i(\bv_i-\bu)$ and fluid pressure $p_i$, where $\varphi_i$ denotes
the porosity of the solid induced by the $i$-th network. For consistency reasons we may assume that
$$
\sum_{i=1}^n \varphi_i = \varphi \in (0,1) \quad \mbox{and} \quad \varphi_i \in (0,1) \quad \mbox{for all }
i=1,2,\ldots,n.
$$

\noindent
The system of two momentum and one mass balance equations for $n$ fluid networks reads
\begin{subequations}\label{eq:MPETini}
\begin{align}
-\divv\sig + \rho\ddot{\bu}  +\sum_{i=1}^{n}\rho_i \ddot{\bw_i} + \sum_{i=1}^{n} \tilde{\alpha_i} \nabla p_i
&= \ \ \tilde{\bf}, \ \quad \text{in } \Omega\times(0,T), \label{eq:MPETini_a} \\
\rho_i \ddot{\bu}+\rho_{m_i} \ddot{\bw_i} + \bK_i^{-1} \dot{\bw_i} + \nabla p_i 
&= -\tilde{\bf}_i, \quad \text{in } \Omega\times(0,T), \quad \mbox{for all } i=1,\cdots,n,\label{eq:MPETini_b} \\
-\tilde{\alpha_i} \dot{\divv \bu} - \dot{\divv \bw_i} -c_{p_i} \dot{p_i} - \sum_{\substack{j=1\\j\neq i}}^{n} \tilde{\beta}_{ij}(p_i-p_j)
&= \ \ 0, \ \, \quad \text{in } \Omega\times(0,T), \quad \mbox{for all } i=1,\cdots,n,\label{eq:MPETini_c},
\end{align}
\end{subequations}
herewith generalizing~\eqref{eqM1}--\eqref{eqM3}, where $\rho:= \sum_{i=1}^{n} \varphi_i\rho_i + (1-\varphi)\rho_s$
again denotes the total density and $\tilde{\bf}$ the body force; The mass densities of the solid and the $i$-th fluid
component are denoted by $\rho_s$ and $\rho_i$, respectively, the Biot-Willis parameter of the $i$-th network by
$\tilde{\alpha_i} \in [\varphi_i ,1]$.
Further, $\tilde{\bf}_i$ is the body force associated with the $i$-th fluid compartment. Moreover, each fluid is characterized
by its effective density $\rho_{m_i} \geq \frac{\rho_i}{\varphi_i}$, cf~\cite{Biot1962}, and viscosity $\eta_i$ resulting in a hydraulic conductivity
$\bK_i := \bm{\kappa}_{i}/\eta_i$ of the $i$-th network, where $\bm{\kappa}_{i}$ denotes its permeability.

Note that the additional term
$
\sum_{j\neq i} \tilde{\beta}_{ij}(p_i-p_j)
$
in \eqref{eq:MPETini_c} models mass exchange between the networks due to pressure differences, cf~\cite{Chou2016afully}\cite{TullyVentikos2011cerebral}.

Applying a symmetrization procedure analogous to the one that has lead to~\eqref{eq:Biot} we obtain 
\begin{subequations}
\begin{align}
-\divv\sig + ((1-\varphi)\rho_s -\sum_{i=1}^{n}\varphi_i\rho_i+\sum_{i=1}^{n}\varphi_i^2\rho_{m_i}) \ddot{\bu}  + \sum_{i=1}^{n}\varphi_i^2\bK_i^{-1}\dot{\bu}  &+\sum_{i=1}^{n}\left((\rho_i  -\varphi_i\rho_{m_i})\ddot{\bv_i}  - \varphi_i\bK_i^{-1}\dot{\bv_i}\right)
\nonumber \\
+\sum_{i=1}^{n}\alpha_i\nabla p_i  &=\bf, \quad \label{eq:MPET_a} \\
(\rho_i  - \varphi_i\rho_{m_i}) \ddot{\bu} - \varphi_i\bK_i^{-1}\dot{\bu}+\rho_{m_i} \ddot{\bv_i} + \bK_i^{-1}\dot{\bv_i} +\nabla p_i 
&= \bg_i, \quad i=1,\cdots,n, \label{eq:MPET_b} \\
-\alpha_i\dot{\divv \bu}  -\dot{\divv \bv_i}  -c_{p_i}\dot{p_i} -\sum_{\substack{j=1\\i\neq j}}^{n} \tilde{\beta}_{ij}(p_i-p_j) 
&= 0, \quad \ i=1,\cdots,n,  \label{eq:MPET_c}
\end{align}\label{eq:MPET}
\end{subequations}
where $0 \leq \alpha_i := \tilde{\alpha_i}-\varphi_i$ and  $\bf:=\tilde{\bf} + \sum_{i=1}^n \varphi_i \tilde{\bf}_i$
and $\bg_i:=-\tilde{\bf}_i$.

Again, the system~\eqref{eq:MPET} can be represented in the form~\eqref{eq:dyn_sys} but now with operators
 $\bM$, $\bD$, $\bL$, right hand side $\bF$, and unknown vector $\by$ given by
\begin{subequations}\label{eq:op_MPET}
\begin{align}
\bM&=
\begin{bmatrix}
((1 \msp - \msp \varphi)\rho_s \msp - \msp \sum_{i=1}^{n} \varphi_i (\rho_i \msp - \msp \varphi_i \rho_{m_i})) I & (\rho_1  \msp - \msp \varphi_1\rho_{m_1}) I 
& \cdots & (\rho_n  \msp - \msp \varphi_n\rho_{m_n}) I & 0  & \cdots & 0 \ \\[1ex]
(\rho_1  - \varphi_1\rho_{m_1}) I & \rho_{m_1} I 
& \cdots & 0 & 0 & \cdots & 0 \ \\
\vdots & \vdots 
& \ddots & \vdots & \vdots & \ddots & \vdots \ \\
(\rho_n  - \varphi_n\rho_{m_1}) I & 0 
& \cdots & \rho_{m_n} I & 0 & \cdots & 0 \ \\[1ex]
0 & 0 
& \cdots & 0 & 0 & \cdots & 0 \ \\
\vdots & \vdots 
& \ddots & \vdots & \vdots & \ddots & \vdots \ \\
0 & 0 & \cdots 
& 0 & 0 & \cdots & 0 \
\end{bmatrix},  \label{eq:M_gen} \\[2ex]
\bD&=\left[\begin{array}{ccccccccccc}
\sum_{i=1}^{n}\varphi_i^2\bK_i^{-1} & -\varphi_1\bK_1^{-1} 
& \cdots & -\varphi_n\bK_n^{-1} & 0
&\cdots&0\\[1ex]
-\varphi_1\bK_1^{-1} & \bK_1^{-1} 
& \cdots & 0  & 0 
&\cdots &0\\
\vdots & \vdots & \ddots & \vdots & \vdots&\ddots &\vdots \\
-\varphi_n\bK_n^{-1} & 0 & \cdots & \bK_n^{-1}& 0&\cdots &0\\[1ex]
-\alpha_1\divv & -\divv & \cdots & 0 & - c_{p1} I &\cdots &0 \\
\vdots & \vdots & \ddots & \vdots  & \vdots & \ddots &\vdots \\
-\alpha_n\divv & 0 & \cdots & -\divv & 0&\cdots&-c_{pn} I \\
\end{array}\right], \label{eq:D_gen} \\[2ex]
\bL&=\left[\begin{array}{ccccccccccc}
- 2\mu\divv\eps -\lambda \nabla\divv & 0 & \cdots & 0 & \alpha_1\nabla & \cdots & \alpha_n\nabla\\[1ex]
0 & 0 & \cdots & 0& \nabla &\cdots &0\\
\vdots & \vdots & \ddots & \vdots & \vdots & \ddots &\vdots\\
0 & 0 & \cdots & 0& 0&\cdots &\nabla\\[1ex]
0 & 0 & \cdots & 0  & -\tilde{\beta}_{11} I &\cdots &\tilde{\beta}_{1n} I \\
\vdots & \vdots & \ddots & \vdots & \vdots & \ddots &\vdots \\
0 & 0 & \cdots & 0 & \tilde{\beta}_{n1} I & \cdots & -\tilde{\beta}_{nn} I \\
\end{array}\right], \label{eq:L_gen} 
\end{align}
\end{subequations}
%
%
\begin{equation}\label{eq:Fr} 
\bF=
\begin{bmatrix}
\bf \\
\bg_1 \\
\vdots \\
\bg_n \\
0 \\
\vdots \\
0
\end{bmatrix},  \qquad
\by=
\begin{bmatrix}
\bu \\
\bv_1 \\
\vdots \\
\bv_n \\
p_1 \\
\vdots \\
p_n
\end{bmatrix} .
\end{equation}
Note that, as before, $(\bM+\bL+\bD)$ is self-adjoint. In the next section, we will use (block) operators
composed of submatrices of $\bM$, $\bL$ and $\bD$. For this reason we define
\begin{equation}\label{op_partitioning}
\bM:=
\begin{bmatrix}
\bM_{11} & \bM_{12} & \bM_{13} \\
\bM_{21} & \bM_{22} & \bM_{23} \\
\bM_{31} & \bM_{32} & \bM_{33}
\end{bmatrix}, \quad
\bD:=
\begin{bmatrix}
\bD_{11} & \bD_{12} & \bD_{13} \\
\bD_{21} & \bD_{22} & \bD_{23} \\
\bD_{31} & \bD_{32} & \bD_{33}
\end{bmatrix}, \quad
\bL:=
\begin{bmatrix}
\bL_{11} & \bL_{12} & \bL_{13} \\
\bL_{21} & \bL_{22} & \bL_{23} \\
\bL_{31} & \bL_{32} & \bL_{33}
\end{bmatrix}
\end{equation}
where the three-by-three partitioning of $\bM$, $\bL$ and $\bD$ corresponds with the partitioning
of the unknown vector $\by$ into $\bu$, $\bv:=(\bv_1^T,\ldots,\bv_n^T)^T$ and $\bp=(p_1,\ldots,p_n)^T$.
The definition of $\bM_{i,j}$, $\bL_{i,j}$ and $\bD_{ij}$ for $1\le i,j \le 3$ then follows from equating
corresponding blocks in~\eqref{op_partitioning} and \eqref{eq:M_gen}--\eqref{eq:L_gen}.

\section{Discretization}\label{sec:discretization}

In this section we present first a second order time discretization method for the dynamic MPET problem and
then recall a family of mixed finite element methods for space discretization that provide mass conservation
in a strong, that is, pointwise, sense.

\subsection{Time discretization}

To start with, consider the equation~\eqref{eq:dyn_sys} with the operators $\bM$, $\bD$, $\bL$, right hand
side $\bF$, and unknown vector $\by$ defined by~\eqref{eq:op_MPET} and \eqref{eq:Fr}.  A second-order
accurate implicit time integration method that can be represented in the form~\eqref{eq:dyn_sys} is the Crank Nicolson method method
. We want to use it in the present context in which we have to resolve the
issue that the operator~$\bM$ is not invertible. 

Let us consider a time interval $[0,T]$, for simplicity partitioned into $n$ equidistant subintervals of
length $\tau$, i.e., $\tau=T/n$.
Then, starting from known initial values for $\bu$, $\bv:=(\bv_1^T,\ldots,\bv_n^T)^T$ and $\bp=(p_1,\ldots,p_n)^T$
at time $t=0$, which we will denote by $\bu^0$, $\bv^0$ and $\bp^0$ and collect in a vector $\by^0$ the time-stepping
scheme we wish to construct should produce a time-discrete approximation of the vector of unknowns $\by$ at time
$t_{k+1}=t_k+\tau$ denoted by $\by^{k+1}$ from the time-discrete approximation $\by^k$ at time $t_k$ by solving an
operator equation of the form
\begin{equation}\label{eq:time_step_problem}
\bA \by^{k+1} = \bG^{k+1}
\end{equation}
where the right hand side $\bG^{k+1}$ is defined in terms of computable quantities at time $t_k$ and $t_{k+1}$,
including, for instance, approximations of the time derivatives $\dot{\bu}$ and $\dot{\bv}$ at time $t_k$
if available.\footnote{Note that the exact values of $\dot{\bu}$ and $\dot{\bv}$ are given at time $t_0=0$.}

Assuming for a moment that we know $\bp$ we can consider a resticted dynamical problem
\begin{equation}\label{eq:dyn_sys_res}
\bar{\bM} \, \ddot{\br} + \bar{\bD}  \dot{\br} + \bar{\bL}  \br = \bar{\bF}
\end{equation}
for the unknowns $\bu$ and $\bv$ which we collect in the vector $\br$, i.e., $\br=(\bu^T,\bv^T)^T$, where
the operators $\bar{\bM}$, $\bar{\bD}$, $\bar{\bL}$ and right hand side $\bar{\bF}$ are defined by
\begin{equation}\label{eq:op_res}
\bar{\bM}:=
\begin{bmatrix}
\bM_{11} & \bM_{12} \\
\bM_{21} & \bM_{22} 
\end{bmatrix}, \quad
\bar{\bD}:=
\begin{bmatrix}
\bD_{11} & \bD_{12} \\
\bD_{21} & \bD_{22} 
\end{bmatrix}, \quad
\bar{\bL}:=
\begin{bmatrix}
\bL_{11} & \bL_{12}  \\
\bL_{21} & \bL_{22} 
\end{bmatrix}, \quad
\bar{\bF}:=
\begin{bmatrix}
\bf \\
\bg
\end{bmatrix} -
\begin{bmatrix}
\bL_{13} \\
\bL_{23}
\end{bmatrix}
\bp .
\end{equation}
 Introducing the new variable $\bs:=\dot{\br}$,
the second-order system~\eqref{eq:dyn_sys_res} can be rewritten in form of the following equivalent first-order system:
\begin{subequations}\label{first_order_sys}
\begin{align}
\dot{\br} &= \bs\label{first_order_sys_a}\\
\dot{\bs} &=  - \bar{\bM}^{-1} \bar{\bD} \bs - \bar{\bM}^{-1} \bar{\bL} \br + \bar{\bM}^{-1} \bar{\bF}
=: \bar{\bH} \label{first_order_sys_b} .
\end{align}
\end{subequations}
Applying Crank-Nicolson method to~\eqref{first_order_sys}
one computes the approximation $\br^{k+1}$ at time $t_{k+1}$ from the system
\begin{subequations}\label{Newmark_sys}
\begin{align}
\br^{k+1} &= \br^{k} + \frac{\tau}{2} (\bs^{k} + \bs^{k+1}) , \label{Newmark_sys_a}\\
\bs^{k+1} &=\bs^{k} + \frac{\tau}{2} (\bar{\bH}(\br^{k},\bs^{k},t_k)+\bar{\bH}(\br^{k+1},\bs^{k+1},t_{k+1}))
=:\bs^{k} + \frac{\tau}{2} (\bar{\bH}^{k}+\bar{\bH}^{k+1}) \label{Newmark_sys_b}.
\end{align}
\end{subequations}
Using the definition of
$\bar{\bH}^{k}$ and $\bar{\bH}^{k+1}$ according to~\eqref{first_order_sys_b} yields
\begin{equation}\label{eq:a_new}
(\bar{\bM} + \frac{\tau}{2}\bar{\bD})\bs^{k+1} =  (\bar{\bM} - \frac{\tau}{2}\bar{\bD}) \bs^{k}
 - \frac{\tau}{2}(\bar{\bL} \br^{k+1}+ \bar{\bL} \br^{k} )
+\frac{\tau}{2}(\bar{\bF}^k +\bar{\bF}^{k+1})
\end{equation}
Next, inserting~\eqref{eq:a_new} in~\eqref{Newmark_sys_a}.
Collecting terms, the time-step equation for the resticted dynamical problem~\eqref{eq:dyn_sys_res}
is given by
\begin{subequations}\label{eq:time_step_eq}
\begin{align}
(\bar{\bM}+\frac{\tau}{2} \bar{\bD}+\frac{\tau^2}{4}\bar{\bL}) \br^{k+1} 
&=\frac{\tau^2}{4}(\bar{\bF}^{k} +\bar{\bF}^{k+1})+ (\bar{\bM}   +\frac{\tau}{2}\bar{\bD} -\frac{\tau^2}{4}\bar{\bL} )\br^k
+\tau\bar{\bM}\bs^k \label{eq:time_step_eq_a}\\
\frac{\tau}{2}\bs^{k+1} - \br^{k+1} &= -\br^{k} - \frac{\tau}{2}\bs^{k} \label{eq:time_step_eq_b}
\end{align}
\end{subequations}
For symmetry reasons,multiplying equation~\eqref{eq:time_step_eq_b} with $(-1)$ and adding it to
equation~\eqref{eq:time_step_eq_a}, then multiplying equation~\eqref{eq:time_step_eq_b} with $\frac{\tau}{2}$, we obtain
\begin{subequations}\label{eq:time_step_eq_sym}
	\begin{align}
	(\bar{\bM}+\textbf{I}+\frac{\tau}{2} \bar{\bD}+\frac{\tau^2}{4}\bar{\bL}) \br^{k+1} -\frac{\tau}{2}\textbf{I}\bs^{k+1} 
	&=\frac{\tau^2}{4}(\bar{\bF}^{k} +\bar{\bF}^{k+1})+ (\bar{\bM}+\textbf{I}   +\frac{\tau}{2}\bar{\bD} -\frac{\tau^2}{4}\bar{\bL} )\br^k
	+\tau(\bar{\bM}+\frac{1}{2}\textbf{I})\bs^k \label{eq:time_step_eq_sym_a}\\
	\frac{\tau^2}{4}\bs^{k+1} - \frac{\tau}{2}\br^{k+1} &= -\frac{\tau}{2}\br^{k} - \frac{\tau^2}{4}\bs^{k} \label{eq:time_step_eq_sym_b}
	\end{align}
\end{subequations}

As a consequence of the presence of $\bp^{k}$ and $\bp^{k+1}$ in~\eqref{eq:time_step_eq_sym_a}, i.e., 
\begin{equation}\label{eq:time_step_coupling}
\bar{\bF}^{k} +\bar{\bF}^{k+1}=
\begin{bmatrix}
\bf^{k} \\
\bg^{k}
\end{bmatrix}
+
\begin{bmatrix}
\bf^{k+1} \\
\bg^{k+1}
\end{bmatrix}
- \begin{bmatrix}
\bL_{13} \\
\bL_{23}
\end{bmatrix}
\bp^{k} 
- \begin{bmatrix}
\bL_{13} \\
\bL_{23}
\end{bmatrix}
\bp^{k+1} ,
\end{equation}
it is not possible to apply~\eqref{eq:time_step_eq_sym} as a stand-alone scheme. This is why
we will couple~\eqref{eq:time_step_eq_sym} to a second time-step equation obtained from the
mass balance equation. 

We use the operators defined in the previous section to rewrite~\eqref{eq:MPET_c} in the form
\begin{equation}\label{mass_balance_MPET}
\bD_{31} \dot{\bu} + \bD_{32} \dot{\bv} + \bD_{33} \dot{\bp} + \bL_{33} \bp = {\bm 0} ,
\end{equation}
or, equivalently,
\begin{equation}\label{mass_balance_MPET_new}
\dot{\tilde{\bp}} = - \bL_{33} \bp.
\end{equation}
where we have introduced the new variable
$\tilde{\bp}:=\bD_{31} \bu + \bD_{32} \bv + \bD_{33} \bp$.
Application of the Crank-Nicolson scheme to~\eqref{mass_balance_MPET_new} results in
\begin{equation}\label{Crank_nicolson_mass_balance}
\tilde{\bp}^{k+1} = \tilde{\bp}^{k} + \frac{\tau}{2} (\dot{\tilde{\bp}}^{k} +\dot{\tilde{\bp}}^{k+1})
= \tilde{\bp}^{k} - \frac{\tau}{2} (\bL_{33} \bp^{k}+\bL_{33} \bp^{k+1}) ,
\end{equation}
which can also be expressed as
\begin{equation}
\bD_{31} \bu^{k+1} + \bD_{32} \bv^{k+1} + (\frac{\tau}{2} \bL_{33} + \bD_{33}) \bp^{k+1}
=\bD_{31} \bu^{k} + \bD_{32} \bv^{k} - (\frac{\tau}{2} \bL_{33} - \bD_{33}) \bp^{k} .\nonumber
\end{equation}
To ensure the Symmetry of all System, multiplying the above equation with $\frac{\tau^2}{4}$ yields
\begin{equation}\label{Crank_Nicolson_mass_balance_new}
\frac{\tau^2}{4}\bD_{31} \bu^{k+1} + \frac{\tau^2}{4}\bD_{32} \bv^{k+1} + \frac{\tau^2}{4}(\frac{\tau}{2} \bL_{33} + \bD_{33}) \bp^{k+1}
=\frac{\tau^2}{4}\bD_{31} \bu^{k} + \frac{\tau^2}{4}\bD_{32} \bv^{k} - \frac{\tau^2}{4}(\frac{\tau}{2} \bL_{33} - \bD_{33}) \bp^{k} .
\end{equation}
The combined scheme is now defined based on~\eqref{eq:time_step_eq}, which we
slightly rearrange in the form
\begin{subequations}\label{Crank_Nicolson_momentum_new}
	\begin{align}
	(\bar{\bM}+\textbf{I}+\frac{\tau}{2} \bar{\bD}+\frac{\tau^2}{4}\bar{\bL}) \br^{k+1} -\frac{\tau}{2}\textbf{I}\bs^{k+1}+&\frac{\tau^2}{4}\begin{bmatrix}
	\bL_{13} \\
	\bL_{23}
	\end{bmatrix}
	\bp^{k+1}
	=
	\frac{\tau^2}{4}\begin{bmatrix}
	\bf^{k} \\
	\bg^{k}
	\end{bmatrix}
	+
	\frac{\tau^2}{4}\begin{bmatrix}
	\bf^{k+1} \\
	\bg^{k+1}
	\end{bmatrix}
	-\frac{\tau^2}{4} \begin{bmatrix}
	\bL_{13} \\
	\bL_{23}
	\end{bmatrix}
	\bp^{k} \nonumber
    \\&+ (\bar{\bM}+\textbf{I}   +\frac{\tau}{2}\bar{\bD} -\frac{\tau^2}{4}\bar{\bL} )\br^k
	+\tau(\bar{\bM}+\frac{1}{2}\textbf{I})\bs^k \label{Crank_Nicolson_momentum_new_a}\\
	\frac{\tau^2}{4}\textbf{I}\bs^{k+1} - \frac{\tau}{2}\textbf{I}\br^{k+1} &= -\frac{\tau}{2}\textbf{I}\br^{k} - \frac{\tau^2}{4}\textbf{I}\bs^{k} \label{Crank_Nicolson_momentum_new_b}
	\end{align}
\end{subequations}
in order to collect all unknown quantities referring to time $t_{k+1}$ on the left hand side
and~\eqref{Crank_Nicolson_mass_balance_new}. It finally can be represented in the  modified
form~\eqref{eq:time_step_problem} where the operator $\bA$ and right hand side vector
$\bG^{k+1}$
are given by
\begin{equation}\label{time_stepping_scheme_op}
\bA:=
\begin{bmatrix}
\bar{\bA}_{11} & \bar{\bA}_{12} & -\frac{\tau}{2}\textbf{I} & \textbf{0} &\frac{\tau^2}{4}\bL_{13} \\
\bar{\bA}_{21} & \bar{\bA}_{22} &\textbf{0}  & -\frac{\tau}{2}\textbf{1} & \frac{\tau^2}{4}\bL_{23} \\
-\frac{\tau}{2}\textbf{I} & \textbf{0}     & \frac{\tau^2}{4}\textbf{I} &\textbf{0} &\frac{\tau^2}{4}\bL_{13} \\
\textbf{0}     & -\frac{\tau}{2}\textbf{I}  &\textbf{0}     & \frac{\tau^2}{4}\textbf{I} & \frac{\tau^2}{4}\bL_{23} \\
\frac{\tau^2}{4}\bD_{31} & \frac{\tau^2}{4}\bD_{32} & \textbf{0}&\textbf{0}&\frac{\tau^3}{8} \bL_{33} + \frac{\tau^2}{4}\bD_{33}
\end{bmatrix},\quad\by=
\begin{bmatrix}
\bu \\
\bv \\
\dot{\bu} \\
\dot{\bv} \\
\bp
\end{bmatrix}
\end{equation}
where
\begin{equation}\label{eq:def_Abar}
\begin{bmatrix}
\bar{\bA}_{11} & \bar{\bA}_{12}   \\
\bar{\bA}_{21} & \bar{\bA}_{22} 
\end{bmatrix} :=
\begin{bmatrix}
{\bM}_{11}+ \textbf{I}+\frac{\tau}{2} {\bD}_{11}+\frac{\tau^2}{4}{\bL}_{11}
&  {\bM}_{12}+\frac{\tau}{2}  {\bD}_{12}+\frac{\tau^2}{4}{\bL}_{12}  \\
 {\bM}_{21}+\frac{\tau}{2}  {\bD}_{21}+\frac{\tau^2}{4}{\bL}_{21}
&  {\bM}_{22}+ \textbf{I}+\frac{\tau}{2}  {\bD}_{22}+\frac{\tau^2}{4}{\bL}_{22} 
\end{bmatrix},
\end{equation}
$\dot{\bv}:=(\dot{\bv}_1^T,\ldots,\dot{\bv}_n^T)^T$
and
\begin{equation}\label{time_stepping_scheme_rhs}
\bG^{k+1}:=
\begin{bmatrix}
\frac{\tau^2}{4}\bf^{k} \msp+\msp \frac{\tau^2}{4}\bf^{k+1}  +  \bar{\bA}_{11} \bu^{k} + \bar{\bA}_{12} \bv^{k} - \frac{\tau^2}{2} (\bL_{11} \bu^{k} \msp+\msp \bL_{12} \bv^{k}) - \frac{\tau^2}{4}\bL_{13} \bp^{k}
+ \tau ((\bM_{11}+\frac{1}{2}\textbf{I}) \dot{\bu}^{k} \msp+\msp \bM_{12} \dot{\bv}^{k})  \\
\frac{\tau^2}{4}\bg^{k} \msp+\msp \frac{\tau^2}{4}\bg^{k+1}  +  \bar{\bA}_{21} \bu^{k} + \bar{\bA}_{22} \bv^{k} - \frac{\tau^2}{2} (\bL_{21} \bu^{k} \msp+\msp \bL_{22} \bv^{k}) - \frac{\tau^2}{4}\bL_{23} \bp^{k}
+ \tau ((\bM_{21}+\frac{1}{2}\textbf{I}) \dot{\bu}^{k} \msp+\msp \bM_{22} \dot{\bv}^{k})  \\
-\frac{\tau}{2}\textbf{I}\bu^k - \frac{\tau^2}{4}\textbf{I}\bud^k\\
-\frac{\tau}{2}\textbf{I}\bv^k - \frac{\tau^2}{4}\textbf{I}\bvd^k\\
\frac{\tau^2}{4}\bD_{31} \bu^{k} + \frac{\tau^2}{4}\bD_{32} \bv^{k} -\frac{\tau^2}{4} (\frac{\tau}{2} \bL_{33} - \bD_{33}) \bp^{k} 
\end{bmatrix} . \
\end{equation}
We see that the right hand is defined in terms of the quantities $\bu^{k}, \bv^{k},\dot{\bu}^{k}, \dot{\bv}^{k} $ and $\bp^{k}$ that are known
from the previous time step, and additionally $\bf^{k}, \bg^{k}$ and  $\bf^{k+1}, \bg^{k+1}$, which can be evaluated at any time moment
due to the known right hand side of~\eqref{eq:MPET}.\\
In summary, we have defined a time-stepping scheme that requires in each time step the solution of an equation
of the form~\eqref{eq:time_step_problem}

with a self-adjoint operator $\bA$ defined in~\eqref{time_stepping_scheme_op},
and we introduce the new abbreviations
\begin{subequations}\label{abreviations}
	\begin{align}
	\g_i &:= -\left((\rho_i  - \varphi_i\rho_{m_i}) - \frac{\tau}{2}\varphi_i K_i^{-1} \right) , \label{gamma_i} \\
	\g_u &:=\left((1-\varphi)\rho_s+1+\sum_{i=1}^{n}\varphi_i \g_i \right),\quad \g_{v,i} :=\frac{\rho_i + \g_i}{\varphi_i} +1\\
	\beta_{ij} &:= \frac{\tau^3}{8}\tilde{\beta_{ij}} , \quad 1 \le i,j \le n, \ i \neq j ,  \quad 
	\beta_{ii} :=\sum_{\substack{j=1\\i\neq j}}^{n}\frac{\tau^3}{8}\tilde{\beta_{ij}} + \frac{\tau^2}{4}c_{p_i} , \quad 1 \le i \le n. \label{beta_ij}
	\end{align}
\end{subequations}
and assumed that $\bK_i = K_i \bm I$.
Then a self-adjoint operator $\bA$ defined in~\eqref{time_stepping_scheme_op}, can be written as follows,
\begin{align*}
\mathcal{A}:=\left[ \begin{array}{ccccccccccccccc}
-\frac{\tau^2}{4}\divv\sig + \g_u & -\g_1 &  \cdots & -\g_n & &-\frac{\tau}{2}& &0&\cdots&0& \frac{\tau^2}{4}\alpha_1\nabla & \cdots  & \frac{\tau^2}{4}\alpha_n\nabla\\\\
-\g_1 & \g_{v,1}& \cdots &0 & &0& &-\frac{\tau}{2}&\cdots&0&\frac{\tau^2}{4}\nabla  & \cdots & 0\\
\vdots&  &\ddots &\vdots & & \vdots& & \vdots&\ddots& \vdots& \vdots & \ddots & \vdots\\
-\g_n & 0& \cdots & \g_{v,n} & &0& &0&\cdots&-\frac{\tau}{2}& 0 & \cdots  & \frac{\tau^2}{4}\nabla\\\\
-\frac{\tau}{2} & 0    &\cdots&0     & &\frac{\tau^2}{4}      & &0     &\cdots&0     & 0    & \cdots & 0\\\\
0     &-\frac{\tau}{2}&\cdots&0          & &0      & &\frac{\tau^2}{4}     &\cdots&0     & 0    & \cdots & 0\\
\vdots&\vdots     &\ddots&\vdots     & &\vdots & &\vdots&\ddots&\vdots&\vdots& \ddots & \vdots\\
0     & 0         &\cdots&-\frac{\tau}{2}& &0      & &0     &\cdots&\frac{\tau^2}{4}     &0     & \cdots & 0\\\\
-\frac{\tau^2}{4}\alpha_1\divv & -\frac{\tau^2}{4}\divv& \cdots &0      & &0     & &0&\cdots&0& -\beta_{11}  & \cdots & \beta_{1n}\\
\vdots        & \vdots&\ddots  &\vdots & &\vdots& &\vdots & \ddots & \vdots& \vdots & \ddots & \vdots\\
-\frac{\tau^2}{4}\alpha_n\divv & 0     & \cdots &-\frac{\tau^2}{4}\divv & &0     & &0&\cdots&0& \beta_{n1}  & \cdots & -\beta_{nn}\\
\end{array}\right] 
\end{align*}
\section{Mass conserving space discretization}
\subsection{Space discretization of continuous Problem }
The weak formulation of system \eqref{Crank_Nicolson_momentum_new} and \eqref{Crank_Nicolson_mass_balance_new} :\newline
Find $(\bu;\bv;\bud;\bvd;\bp) \in \bUd\times\bU\times \bVd\times\bV\times \bP$, 
such that for any $(\bw;\bz;\bwd;\bzd;\bq)  \in \bUd\times\bU\times \bVd\times\bV\times \bP$ then  

\begin{subequations}
\begin{align}
\frac{\mu\tau^2}{2}(\eps(\bu),\eps(\bw)) +\frac{\lambda\tau^2}{4}(\divv \bu,\divv \bw)+ \g_u(\bu,\bw)  -\frac{\tau}{2}(\bud,\bw) +(\bar{\bA}_{12}\bv,\bw)-\frac{\tau^2}{4}(\ba\bp,\Divu \bw)&= (\bG_1,\bw), \\
(\bar{\bA}_{21}\bu,\bz)         + (\bar{\bA}_{22}\bv,\bz) - \frac{\tau}{2}(\bvd,\bz) -\frac{\tau^2}{4}( \bp,\Divv \bz)  &= (\bG_2,\bz),\\
-\frac{\tau}{2}(\bu,\bwd) + \frac{\tau^2}{4}(\bud,\bwd) &= (\bG_3,\bwd),\\
-\frac{\tau}{2}(\bv,\bzd)+\frac{\tau^2}{4}(\bvd,\bzd) &=(\bG_4,\bzd),\\
-\frac{\tau^2}{4}(\ba\Divu \bu,\bq)  - \frac{\tau^2}{4}(\Divv \bv,\bq)  + ((\frac{\tau^3}{8} \bL_{33} + \frac{\tau^2}{4}\bD_{33})\bp,\bq)  &= (\bG_5,\bq),
\end{align}\label{eq:MPET weak}
\end{subequations}
where $$\Divv \bv = \left(\begin{array}{c}
\divv \bv_1\\
\vdots\\
\divv \bv_n
\end{array}\right)\text{~~for all } \bv \in \bU,\quad \Divu \bu = \left(\begin{array}{c}
\divv \bu\\
\vdots\\
\divv \bu
\end{array}\right)\text{~~for all } \bu \in \bUd,\ba :=
\begin{bmatrix}
\alpha_1  &0 &  \cdots &0 \\
0 & \alpha_2& \cdots & 0 \\
\vdots &\vdots & \ddots & \vdots \\
0 &0 & \cdots&  \alpha_n \\
\end{bmatrix}$$
Consider the Hilbert spaces $\bUd = H_0^1(\Omega)^d,~\bVd=H_0(\divv, \Omega),~ \bP = (L_0^2(\Omega))^n$ and $\bU,\bV = (H_0(\divv, \Omega))^n$ with parameter-dependent
norms $\|\cdot\|_{\bUd\times\bU\times\bVd\times\bV}, \|\cdot\|_{\bP}$ induced by the inner product

\begin{subequations}
\begin{align}
((\bu,\bv,\bud,\bvd),(\bw,\bz,\bwd,\bzd))_{\bUd\times\bU\times\bVd\times\bV} &= \frac{\mu\tau^2}{2}(\eps(\bu),\eps(\bw)) +\frac{\lambda\tau^2}{4}(\divv \bu,\divv \bw) + (\Lambda_{uv} \left(\begin{array}{c}
\bu \\\bv \\ \bud\\ \bvd
\end{array}\right),\left(\begin{array}{c}
\bw \\\bz\\ \bwd\\ \bzd
\end{array}\right))\nonumber\\& + \frac{\tau^2}{4}(\Lambda^{-1}\left(\Divv \bv + \ba \Divu \bu\right),\Divv \bz+ \ba \Divu \bw),\\
(\bp,\bq)_{\bP}&=\frac{\tau^2}{4}(\Lambda \bp,\bq)
\end{align}\label{norm}
\end{subequations}
where \begin{align*}
\Lambda &:= \Lambda_1 + \Lambda_2 + \Lambda_3,\quad \Lambda_1:=-(\frac{\tau}{2} \bL_{33} + \bD_{33}),\quad\Lambda_2 := \frac{\tau^2}{4}\bar{\bA}_{22}^{-1}
,\qquad\Lambda_3 := \frac{\tau^2}{4\gamma}\ba\Lambda_4\ba,\quad \gamma := \max{\lbrace \frac{\tau^2\mu}{2},\frac{\tau^2\lambda}{4},\g_u}\rbrace
\end{align*}
\begin{align*}
\Lambda_4 :=
\begin{bmatrix}
1 &1 & \cdots& \cdots & 1 \\
1 & 1& \cdots & \cdots& 1 \\
\vdots &\vdots & \ddots& \cdots & \vdots \\
1 &1 & \cdots& \cdots & 1 \\
\end{bmatrix},\qquad
&\Lambda_{uv}:=\begin{bmatrix}
\g_u            & \bar{\bA}_{12} & -\frac{\tau}{2}\textbf{I} & \textbf{0} \\
\bar{\bA}_{21} & \bar{\bA}_{22} & \textbf{0} & -\frac{\tau}{2}\textbf{I}\\
 -\frac{\tau}{2}\textbf{I} & \textbf{0} & \frac{\tau^2}{4}\textbf{I}& \textbf{0}\\
  \textbf{0} & -\frac{\tau}{2}\textbf{I}& \textbf{0}& \frac{\tau^2}{4}\textbf{I}
\end{bmatrix}
\end{align*}
The matrices $\Lambda_1$ and $\Lambda_{uv}$ are positive semi-definit, because they are sum of positive matrices.\\
From System of equations \eqref{eq:MPET weak} introduce the bilinear form
\begin{align}	\mathcal{A}((\bu;\bv;\bud;\bvd;\bp),(\bw,\bz;\bwd;\bzd;\bq))&=
\frac{\mu\tau^2}{2}(\eps(\bu),\eps(\bw)) +\frac{\lambda\tau^2}{4}(\divv \bu,\divv \bw)+ (\Lambda_{uv} \left(\begin{array}{c}
\bu \\\bv \\ \bud\\ \bvd
\end{array}\right),\left(\begin{array}{c}
\bw \\\bz\\ \bwd\\ \bzd
\end{array}\right)) \nonumber
\\&-\frac{\tau^2}{4}(\bp,\ba\Divu \bw+\Divv \bz) -\frac{\tau^2}{4}(\ba\Divu \bu+\Divv \bv,\bq)  - \frac{\tau^2}{4}(\Lambda_1 \bp,\bq)\label{Bilinear-form}
\end{align}
\subsection{the space discretization of discrete Problem }
\subsubsection{Preliminaries and notation}
Let $\mathcal{T}_h$ be a shape-regular triangulation of mesh-size $h$ of
the domain $\Omega$ into triangles $\{T\}$ and define the set of all interior edges (or faces) of  $\mathcal{T}_h$ by $\mathcal{E}_h^{I}$ and the set of all boundary edges (or faces) by $\mathcal{E}_h^{B}$. Let
$\mathcal{E}_h=\mathcal{E}_h^{I}\cup \mathcal{E}_h^{B}$.

For $s\geq 1$, we introduce the spaces

$$
H^s(\mathcal{T}_h)=\{\phi\in L^2(\Omega), \mbox{ such that } \phi|_T\in H^s(T) \mbox{ for all } T\in \mathcal{T}_h \}.
$$

We further define some trace operators.Denote by $e = \partial T_1 \cap \partial T_2$ the common boundary (interface) of two subdomains $T_1$ and $T_2$ in $\mathcal{T}_h$, and by $\bn_1$ and $\bn_2$, the unit normal vectors to $e$ that point to the exterior of $T_1$ and $T_2$, correspondingly.

For any  $e
\in \mathcal{E}_h^{I}$
and $q\in H^1(\mathcal{T}_h)$,$\bv \in H^1(\mathcal{T}_h)^d$ and $\bm \tau \in H^1(\mathcal{T}_h)^{d\times d}$, the averages are defined as
\begin{equation*}
\begin{split}
\{\bv\} &=\frac{1}{2}(\bv|_{\partial T_1\cap e}\cdot \bn_1-\bv|_{\partial T_2\cap e}\cdot \bn_2), \quad 
\{\bm \tau\}=\frac{1}{2}(\bm \tau|_{\partial T_1\cap
	e} \bn_1-\bm \tau|_{\partial T_2\cap e} \bn_2),
\end{split}
\end{equation*}
and the jumps are given by
\begin{equation*}
[q]=q|_{\partial T_1\cap e}-q|_{\partial T_2\cap e},\quad
[\bv]=\bv|_{\partial T_1\cap e}-\bv|_{\partial T_2\cap e}.
\end{equation*}

When $e \in  \mathcal{E}_h^{B}$, then the above quantities are defined as
$$
\{\bv\}=\bv |_{e}\cdot \bn,\quad
\{\bm \tau\}=\bm \tau|_{e}\bn, \quad
[q]=q|_{e}, ~~ [\bv]=\bv|_{e}.
$$
If $\bn_T$ is the outward unit normal to $\partial T$, it is easy to show that
, for $\bm \tau \in H^1(\Omega)^{d\times d}$ and for all $\bv\in H^1(\mathcal{T}_h)^d$, we have
\begin{equation}\label{eq:69}
\sum_{T\in \mathcal{T}_h}\int_{\partial T}(\bm \tau\bn_T)\cdot  \bv ds=\sum_{e\in \mathcal{E}_h}\int_{e}\{\bm \tau\}\cdot [\bv] ds.
\end{equation}

\subsubsection{DG discretization}\label{DG discretization}
The finite element spaces we consider are denoted by
\begin{align*}
\bUd_h&=\{\bu \in H(\divv ;\Omega):\bu|_T \in \bUd(T),~T \in \mathcal{T}_h;~ \bu \cdot
\bn=0~\hbox{on}~\partial \Omega\},
\\[1ex]
\bU_{i,h}&=\{\bv \in H(\divv ;\Omega):\bv|_T \in \bU_i(T),~T \in \mathcal{T}_h;~ \bv \cdot
\bn=0~\hbox{on}~\partial \Omega\},~~i=1,\cdots,n,
\\
\bVd_h&=\{\tilde{\bu} \in H(\divv ;\Omega):\tilde{\bu}|_T \in \bVd(T),~T \in \mathcal{T}_h;~ \tilde{\bu} \cdot
\bn=0~\hbox{on}~\partial \Omega\},
\\
\bV_{i,h}&=\{\tilde{\bv} \in H(\divv ;\Omega):\tilde{\bv}|_T \in \bV_i(T),~T \in \mathcal{T}_h;~ \tilde{\bv} \cdot
\bn=0~\hbox{on}~\partial \Omega\},~~i=1,\cdots,n,
\\
P_{i,h}&=\{p \in L^2(\Omega):p|_T \in Q_i(T),~T \in \mathcal{T}_h; ~\int_{\Omega}p dx=0\},~~i=1,\cdots,n.
\end{align*}
The discretization we analyze in the present context define the local spaces $\bUd(T)/\bU_i(T)/\bVd(T)/\bV_i(T)/Q_i(T)$ via $BDM_l(T)/$ $BDM_l(T)/ RT_{l-1}(T)/RT_{l-1}(T)/P_{l-1}(T)$, or
$BDFM_l(T)/BDFM_l(T)/RT_{l-1}(T)$ $/RT_{l-1}(T)/P_{l-1}(T)$ for $l\geq 1$. Note that, for each of these choices, the important condition $\divv  \bUd(T)=\divv  \bU_i(T)=\divv  \bVd(T)=\divv  \bV_i(T)=Q_i(T)$ is satisfied, cf~\cite{cockburn2007note}\cite{HongKraus2017parameter}\cite{Hong2018conservativeMPET}.

Note that the normal component of any $\bu\in\bUd_h$ is continuous on the internal edges and vanishes on the boundary edges. Then, for all $e \in \mathcal{E}_h$nd for all$ \bm \tau \in H_1(\mathcal{T})^d, \bu \in \bUd_h$ it holds
\begin{equation}\label{eq:72}
\int_e[\bm u_n]\cdot \bm \tau ds=0,\quad\mbox{implying that}\quad \int_e[\bm u]\cdot \bm \tau ds=\int_e[\bm u_t]\cdot \bm \tau ds,
\end{equation}
where $\bu_n$ and $\bu_t$ denote the normal and tangential component of $\bu$, respectively.

Similar to the continuous problem, we denote
\begin{align*}
\bv_h^T&=(\bv_{1,h}^T, \cdots \bv_{n,h}^T),\qquad \bvd_h^T=(\bvd_{1,h}^T, \cdots \bvd_{n,h}^T),\qquad \bp_h^T=(p_{1,h},\cdots, p_{n,h}),\qquad
\bz_h^T=(\bz_{1,h}^T, \cdots \bz_{n,h}^T),
\\
\bzd_h^T&=(\bzd_{1,h}^T, \cdots \bzd_{n,h}^T),\qquad
\bq_h^T=(q_{1,h},\cdots, q_{n,h}),\qquad \bV_h=\bV_{1,h}\times\cdots\times\bV_{n,h}, \qquad\bP_h= P_{1,h}\times\cdots\times P_{n,h}.
\end{align*}
With this notation at hand, the discretization of the variational problem~\eqref{eq:MPET weak} is given as follows:

Find $(\bu_h; \bv_h;\bud_h;\bvd_h; \bp_h,)~\in~\bUd_h\times \bU_h\times\bVd_h\times \bV_h\times \bP_h$ such that, for any $(\bw_h; \bz_h;\bwd_h; \bzd_h; \bq_h,) \in \bUd_h\times \bU_h\times\bVd_h\times \bV_h\times \bP_h$ 
\begin{subequations}
\begin{align}
\frac{\mu\tau^2}{2}a_h(\bu_h,\bw_h) +\frac{\lambda\tau^2}{4}(\divv \bu_h,\divv \bw_h)+ \g_u(\bu_h,\bw_h)  -\frac{\tau}{2}(\bud_h,\bw_h) +(\bar{\bA}_{12}\bv_h,\bw_h)	
-\frac{\tau^2}{4}(\ba\bp_h,\Divu \bw_h)&= (\bG_1,\bw_h), \\
(\bar{\bA}_{21}\bu_h,\bz_h)         + (\bar{\bA}_{22}\bv_h,\bz_h) - \frac{\tau}{2}(\bvd_h,\bz_h) -\frac{\tau^2}{4}( \bp_h,\Divv \bz_h)  &= (\bG_2,\bz_h),\\
-\frac{\tau}{2}(\bu_h,\bwd_h) + \frac{\tau^2}{4}(\bud_h,\bwd_h) &= (\bG_3,\bwd_h),\\
-\frac{\tau}{2}(\bv_h,\bzd_h)+\frac{\tau^2}{4}(\bvd_h,\bzd_h) &=(\bG_4,\bzd_h),\\
-\frac{\tau^2}{4}(\ba\Divu \bu_h,\bq_h)  - \frac{\tau^2}{4}(\Divv \bv_h,\bq_h)  - \frac{\tau^2}{4}(\Lambda_1\bp_h,\bq_h)  &= (\bG_5,\bq_h),
\end{align}\label{a2dis}
\end{subequations}
where
\begin{eqnarray}
a_h(\bm u,\bm w)&=&\label{78}
\sum _{T \in \mathcal{T}_h} \int_T \eps(\bm{u}) :
\eps(\bm{w}) dx-\sum_{e \in \mathcal{E}_h} \int_e \{\eps(\bm{u})\} \cdot [\bm w_t] ds-\sum _{e \in \mathcal{E}_h} \int_e \{\eps(\bm{w})\} \cdot [\bm u_t]ds+\sum _{e
	\in \mathcal{E}_h} \int_e \eta h_e^{-1}[ \bm u_t] \cdot [\bm w_t] ds,
\end{eqnarray}
$\eta $ is a stabilization parameter independent of all parameters , the network scale $n$ and the mesh size $h$.\\
For any $\bu \in \bUd_h$, we introduce the mesh-dependent norms:
\begin{align*}
\|\bu\|_h^2=\sum _{T \in \mathcal{T}_h} 
\|\eps(\bu)\|_{0,T}^2+\sum _{e \in \mathcal{E}_h} h_e^{-1}\|[ \bu_t]\|_{0,e}^2, \qquad
\|\bu\|_{1,h}^2=\sum _{T \in \mathcal{T}_h} \|\nabla\bu\|_{0,T}^2+\sum _{e \in \mathcal{E}_h} h_e^{-1}\|[ \bu_t]\|_{0,e}^2 .
\end{align*}
the "DG"-norm
\begin{equation}\label{DGnorm}
\|\bu\|^2_{DG}=\sum _{T \in \mathcal{T}_h} \|\nabla\bu\|_{0,T}^2+\sum _{e \in \mathcal{E}_h} h_e^{-1}\|[ \bu_t]\|_{0,e}^2+\sum _{T \in \mathcal{T}_h}h_T^2|\bu|^2_{2,T},
\end{equation}
and, finally, the mesh-dependent norm $\|\bm (\cdot;\cdot;\cdot;\cdot) \|_{\bUd_h\times\bU\times\bVd\times\bV}$ by
\begin{equation}\label{U_hnorm}
\|(\bu;\bv;\bud;\bvd)\|^2_{\bUd_h\times\bU\times\bVd\times\bV}= \frac{\mu\tau^2}{2}\|\bu\|^2_{DG} +\frac{\tau^2\lambda}{4}\|\divv \bu\|^2 + \|\Lambda_{uv}^{\frac{1}{2}} \left(\begin{array}{c}
\bu \\\bv \\\bud \\\bvd
\end{array}\right)\|^2 + \frac{\tau^2}{4}\|\Lambda^{-\frac{1}{2}}\left(\Divv \bv + \ba \Divu \bu\right)\|^2.
\end{equation}
We now summarize several results on well-posedness and approximation
properties of the DG formulation:

\begin{itemize}
\item From the discrete version of Korn's inequality we have that the norms $\|\cdot\|_{DG}$,
$\|\cdot\|_h$, and $\|\cdot\|_{1,h}$ are equivalent on $\bUd_h$, namely,
\begin{subequations}\label{Korninequality}
\begin{align}
\|\bu\|_{DG}\eqsim  \|\bu\|_h\eqsim\|\bu\|_{1,h},\quad\mbox{for all}\quad~\bu \in \bUd_h .\\
\|\bu\|_{DG}^2\leq c_0\|\bu\|_{1,h}^2,\qquad	\|\bu\|_h^2\geq c_1\|\bu\|_{DG}^2
\end{align}
\end{subequations}
	
\item The bilinear form $a_h(\cdot,\cdot)$,
introduced in~\eqref{78} is continuous and we have
\begin{eqnarray}\label{continuity:a_h}
|a_h(\bm u,\bm w)|&\leq& c_2\| \bm u  \|_{DG}  \| \bm w  \|_{DG},\quad\mbox{for all}\quad \bm u,~\bm w\in H^2(\mathcal{T}_h)^d.
\end{eqnarray}
\item  The discrete Poincare inequality, cf~\cite{arnold1982interior}
\begin{align}
\|\bu\|^2 \leq c_3 \|\bu\|_{1,h}^2,\quad\mbox{for all}\quad~\bu \in \bUd_h.\label{Poincare}
\end{align}
\item For our choice of the finite element spaces $\bUd_h, \bU_h$ and {$\bm P_h$} we
have the following inf-sup conditions,
\begin{equation}
\inf_{{\bq_h\in \bP_h}}\sup_{\bu_h\in \bUd_h}\frac{(\divv\bu_h,\sum\limits_{i=1}^n q_{i,h})}{\|\bu_h\|_{1,h}\|\sum\limits_{i=1}^n q_{i,h}\|}\geq \beta_{s,h},\,\,\,\inf_{q_{i,h}\in P_{i,h}}\sup_{\bv_{i,h}\in \bU_{i,h}}\frac{(\divv\bv_{i,h},q_{i,h})}{\|\bv_{i,h}\|_{\divv}\|q_{i,h}\|}\geq \beta_{v,h},\label{divdis}
\end{equation}
where $\beta_{s,h}$ and $\beta_{v,h}$ are positive constant independent of all parameters, the network scale $n$ and the mesh size $h$.
\item The coercivity of $a_h(\cdot,\cdot)$ 
\begin{equation}\label{coercivity:a_h}
a_h(\bm{u}_h,\bm{u}_h)\geq \alpha_a \|\bm{u}_h\|^2_h,\quad\mbox{for all}\quad~\bm{u}_h\in\bUd_h,
\end{equation}
where $\alpha_a$ is a positive constant independent of all parameters, the network scale $n$ and the mesh size $h$.
\end{itemize}

Related to the discrete problem~\eqref{a2dis} and from the definition of the matrix $\Lambda_{uv}$, we define the bilinear
form
\begin{align}
\mathcal A_h((\bu_h; \bv_h;\bud_h; \bvd_h; \bp_h),(\bw_h; \bz_h;\bwd_h; \bzd_h; \bq_h))&=\frac{\mu\tau^2}{2}a_h(\bu_h,\bw_h) +\frac{\tau^2\lambda}{4}(\divv \bu_h,\divv \bw_h)+ (\Lambda_{uv} \left(\begin{array}{c}
\bu_h \\\bv_h\\ \bud_h \\ \bvd_h
\end{array}\right),\left(\begin{array}{c}
\bw_h \\\bz_h \\ \bwd_h \\ \bzd_h
\end{array}\right))
\nonumber\\ &-\frac{\tau^2}{4}(\bp_h,\ba\Divu \bw_h +\Divv \bz_h)-\frac{\tau^2}{4}(\ba\Divu \bu_h+\Divv \bv_h,\bq_h)  - \frac{\tau^2}{4}(\Lambda_1 \bp_h,\bq_h)\label{dis-Bilinear-form}
\end{align}

\section{Stability analysis}\label{sec:stability}

\subsection{Stability of the time-discrete problem}
The main result of this section is a proof of the uniform well-posedness, of problem \eqref{eq:MPET weak} under the norms
induced by~\eqref{norm}. 
Before we study the full dynamic MPET equations, we recall the following well known results, cf~\cite{Boffi2013mixed}\cite{Brezzi1974existence}, and two help Lemmas:
\begin{theorem3}
	There exists a constant $\beta_v > 0$ such that
	\begin{align}
	\inf_{q\in P_i} \sup_{\bv\in \bV_i} \frac{(\divv \bv, q)}{\|\bv\|_{\divv}\|q\|} \geq \beta_v,~i=1,\dots,n.
	\end{align}\label{div}
\end{theorem3}
\begin{theorem3}
	There exists a constant $\beta_s > 0$ such that
	\begin{align}
	\inf_{q\in P} \sup_{\bu\in \bUd} \frac{(\divv \bu, q)}{\|\bu\|_1\|q\|} \geq \beta_s
	\end{align}\label{grad}
\end{theorem3}
\begin{theorem3}
	the determinant of the following matrix
	$$A:=\begin{bmatrix}
	-b_1&-b_2&\cdots&\cdots&-b_n\\
	a&0&\cdots&0&0\\
	\vdots&\ddots &&\vdots&\vdots\\
	\vdots&  &\ddots&\vdots&\vdots\\
	0& \cdots& \cdots&a&0\\
	\end{bmatrix}_{n\times n}  \text{ is }
	\det(A)= (-1)^n\cdot a^{n-1}\cdot b_n$$
\end{theorem3}
\begin{proof}
By using induction method. For $n=1$ we have $\det(A)= -b_1$. 
Assume the induction hypothesis is true for $(n-1)$ and we proof for $n$.
By using the Laplace's formula for the last row it follow
{
$$\det(A) = (-1)^{n+n-1}\cdot a\cdot
\begin{vmatrix}
-b_1&-b_2&\cdots&\cdots&-b_n\\
a&0&\cdots&0&0\\
\vdots&\ddots &&\vdots&\vdots\\
\vdots&  &\ddots&\vdots&\vdots\\
0& \cdots& \cdots&a&0\\
\end{vmatrix}_{(n-1)\times (n-1)} = - a\left((-1)^{n-1}a^{n-2}b_n\right)=(-1)^n\cdot a^{n-1}\cdot b_n$$
	}
\end{proof}
\begin{theorem3}
the determinant of the following matrix
$$B:=\begin{bmatrix}
c&-b_1&-b_2&\cdots&b_n\\
-b_1&a&\cdots&\cdots&0\\
-b_2&\vdots&\ddots &&\vdots\\
\vdots&\vdots&  &\ddots&\vdots\\
-b_n&	0& \cdots& \cdots&a\\
\end{bmatrix}_{(n+1)\times (n+1)} \text{is } \det(B)=a^{n-1}\left(a\cdot c-\sum_{i=1}^{n}b_i^2\right)$$\label{det}
\end{theorem3}
\begin{proof}
By using induction method. For $n=1$ we have $\det(A)= a\cdot c -b_1^2$. 
Assume the induction hypothesis is true $n$ and we proof for $(n+1)$.
By using the Laplace's formula for the last row it follow
{
\begin{align*}
\det(B) &= (-1)^{n+1+1}\cdot (-b_n)\cdot
\begin{vmatrix}
-b_1&-b_2&\cdots&\cdots&-b_n\\
a&0&\cdots&0&0\\
\vdots&\ddots &&\vdots&\vdots\\
\vdots&  &\ddots&\vdots&\vdots\\
0& \cdots& \cdots&a&0\\
\end{vmatrix}_{n\times n} + (-1)^{2n+2}\cdot a\cdot\begin{vmatrix}
c&-b_1&-b_2&\cdots&b_{n-1}\\
-b_1&a&\cdots&\cdots&0\\
-b_2&\vdots&\ddots &&\vdots\\
\vdots&\vdots&  &\ddots&\vdots\\
-b_{n-1}&	0& \cdots& \cdots&a\\
\end{vmatrix}_{n\times n} 
\\&= (-1)^{n+2}\cdot (-b_n)\cdot(-1)^n\cdot a^{n-1}\cdot b_n + a\cdot a^{n-2}\left(a\cdot c-\sum_{i=1}^{n-1}b_i^2\right)=-a^{n-1}\cdot b_n^2 + a^{n-1}\left(a\cdot c-\sum_{i=1}^{n-1}b_i^2\right)
\end{align*}
	}
\end{proof}
The following theorem shows the boundedness of $\mathcal{A}((.;.;.;.;.),(.;.;.;.;.))$ in the norm induced by~\eqref{norm}:
\begin{theorem}
There exists a constant $C_b$ independent of  all parameters and the network scale~$n$, such that  for any

 $(\bu;\bv;\bud;\bvd;\bp)\in \bUd\times\bU\times\bVd\times\bV\times \bP, (\bw,\bz;\bwd;\bzd;\bq)\in \bUd\times\bU\times\bVd\times\bV\times \bP$
\begin{equation}
|\mathcal{A}((\bu;\bv;\bud;\bvd;\bp),(\bw;\bz;\bwd;\bzd;\bq))|\leq C_b (\|(\bu,\bv,\bud,\bvd)\|_{\bUd\times\bU\times\bVd\times\bV}+\|\bp\|_{\bP})\cdot(\|(\bw,\bz,\bwd,\bzd)\|_{\bUd\times\bU\times\bVd\times\bV}+\|\bq\|_{\bP}).\label{bound}
\end{equation}\label{theorem-bound} 
\end{theorem}
\begin{proof}
By applying Cauchy-Schwarz inequality on the bilinear form \eqref{Bilinear-form} we obtain
\begin{align*}
\mathcal{A}((\bu;\bv;\bp),&(\bw;\bz;\bq))\leq \frac{\mu\tau^2}{2}\|\eps(\bu)\|\cdot\|\eps(\bw))\| +\frac{\tau^2\lambda}{4}\|\divv \bu\|\cdot\|\divv \bw\|+ \|\Lambda_{uv}^{\frac{1}{2}} \left(\begin{array}{c}
\bu \\\bv \\ \bud\\ \bvd
\end{array}\right)\|\cdot\|\Lambda_{uv}^{\frac{1}{2}}\left(\begin{array}{c}
\bw \\\bz\\ \bwd\\ \bzd
\end{array}\right)\| \\&+\frac{\tau^2}{4}\|\Lambda^{\frac{1}{2}}\bp\|\cdot\|\Lambda^{-\frac{1}{2}}(\ba\Divu \bw+\Divv \bz)\| 
 +\frac{\tau^2}{4}\|\Lambda^{-\frac{1}{2}}(\ba\Divu \bu+\Divv \bv)\|\cdot\|\Lambda^{\frac{1}{2}}\bq\|  + \frac{\tau^2}{4}\|\Lambda_1^{\frac{1}{2}} \bp\|\cdot\|\Lambda_1^{\frac{1}{2}}\bq\|
\end{align*}
We obtain \eqref{bound}, by applying again Cauchy-Schwarz inequality. 
\end{proof}
The following theorem shows the inf-sup-condition (LBB) of $\mathcal{A}((.;.;.;.;.),(.;.;.;.;.))$ in the norm induced by~\eqref{norm}:
\begin{theorem}
There exists a constant $\omega> 0$ independent of all parameters and the network scale $n$, such that
\begin{align*}
\inf_{\substack {(\bu;\bv;\bud;\bvd;\bp)\\\in \bUd\times \bU\times \bVd\times \bV\times\bP}}\sup_{\substack {(\bw;\bz;\bwd;\bzd;\bq)\\\in   \bUd\times \bU\times \bVd\times \bV\times\bP}}\frac{\mathcal{A}((\bu;\bv;\bud;\bvd;\bp),(\bw;\bz;\bwd;\bzd;\bq))}{(\|(\bu;\bv;\bud;\bvd)\|_{ \bUd\times \bU\times \bVd\times \bV}+\|\bp\|_{\bP})\cdot(\|(\bw;\bz;\bwd;\bzd)\|_{ \bUd\times \bU\times \bVd\times \bV}+\|\bq\|_{\bP})} \geq \omega.
\end{align*}\label{inf-sup}
\end{theorem}
\begin{proof}
For any $(\bu;\bv;\bud;\bvd;\bp)\in \bUd\times\bU\times \bVd\times\bV\times \bP$, by Lemma \ref{div}, there exist
\begin{align}
\bv_0 &\in \bU ~~\hbox{such that}~~\Divv\bv_0 = \frac{\tau}{2}\bar{\bA}_{22}^{-\frac{1}{2}}\bp ~~\hbox{and}~~ \|\bv_0\|_{\divv} \leq { \beta_{v}^{-1}\|\Lambda_2^{-\frac{1}{2}}\bp\|} \label{psi},
\end{align}
and by Lemma \ref{grad}, there exists
\begin{align}
\bu_0 &\in \bUd~~\hbox{such that}~~ 
\Divu \bu_0 = \frac{\tau}{2\sqrt{\gamma}} \Lambda_4\ba\bp ,~~\|\bu_0\|_1 \leq \frac{\tau\beta_{s}^{-1}}{2\sqrt{\gamma}}\|\Lambda_4^{\frac{1}{2}}\ba\bp\|= \beta_s^{-1}\|\Lambda_3^{\frac{1}{2}}\bp\|.\label{u0}
\end{align}
Choose
\begin{align}
\bw = \delta \bu - \frac{\tau}{2\sqrt{\gamma}}\bu_0,\quad \bz = \delta \bv -  \frac{\tau}{2}\bar{\bA}_{22}^{-\frac{1}{2}} \bv_0, ~~ \bwd=\delta\bud, ~~\bzd=\delta\bvd, ~~\bq = -\delta \bp - \frac{\tau^2}{4}\Lambda^{-1}(\Divv \bv+\ba\Divu \bu),\label{wzq}
\end{align}
where $\delta$ is a positive constant to be determinant later. Before we verify the boundedness of $(\bw;\bz;\bwd;\bzd;\bq)$ by $(\bu;\bv;\bud;\bvd;\bp)$. we try to estimate $\|\Lambda_{uv}^{\frac{1}{2}} \left(\begin{array}{c}
\frac{\tau}{2\sqrt{\gamma}}\bu_0 \\  \frac{\tau}{2}\bar{\bA}_{22}^{-\frac{1}{2}} \bv_0 \\\textbf{0} \\\textbf{0}
\end{array}\right)\|^2$:	
\begin{align}
\|\Lambda_{uv}^{\frac{1}{2}} \left(\begin{array}{c}
\frac{\tau}{2\sqrt{\gamma}}\bu_0 \\  \frac{\tau}{2}\bar{\bA}_{22}^{-\frac{1}{2}} \bv_0 \\\textbf{0} \\\textbf{0}
\end{array}\right)&\|^2
=\frac{\tau^2}{4}\left(
\underbrace{\begin{bmatrix}
c &\frac{1}{\sqrt{\gamma}}\bar{\bA}_{12}\bar{\bA}_{22}^{-\frac{1}{2}}\\
(\frac{1}{\sqrt{\gamma}}\bar{\bA}_{12}\bar{\bA}_{22}^{-\frac{1}{2}})^T&\textbf{I}\end{bmatrix}}_{:=G}
\left(\begin{array}{c}
\bu_0 \\\bv_0
\end{array}\right),\left(\begin{array}{c}
\bu_0 \\\bv_0
\end{array}\right)\right) \nonumber\\
&\leq \frac{\tau^2}{4}\lambda_{\max}(G)(\|\bu_0\|^2+\|\bv_0\|^2) \underbrace{\leq}_{\eqref{u0},\eqref{psi}} \frac{\tau^2}{4}\lambda_{\max}(G)\left(\beta_s^{-2}\|\Lambda_3^{\frac{1}{2}}\bp\|^2+\beta_{v}^{-2}\|\Lambda_2^{-\frac{1}{2}}\bp\|^2\right) \label{g}
\end{align}
where $c:= \frac{\g_u}{\gamma}\leq 1, $ and let $ -b_i:=(\frac{1}{\sqrt{\gamma}}\bar{\bA}_{12}\bar{\bA}_{22}^{-\frac{1}{2}})_i=-\g_i\sqrt{\frac{1}{\g_{v,i}}}\frac{1}{\sqrt{\gamma}}, i=1,\cdots,n, $ then 
\begin{align}
\sum_{i=1}^{n}b_i^2=\sum_{i=1}^{n}\left(\g_i^2\frac{\varphi_i}{\rho_i+\varphi_i+\g_i}\frac{1}{\gamma}\right) \leq \sum_{i=1}^{n}\left(\frac{\g_i \varphi_i}{(1-\varphi)\rho_s +1+ \left(\sum_{i=1}^{n}\varphi_i \g_i\right)}\right)\leq 1 \label{sumb}
\end{align}
To find  the eigenvalues of the matrix $G$, we using Lemma \ref{det} :
\begin{align*}
\det(G-\lambda I)&=\begin{vmatrix}
c-\lambda &-b_1 & \cdots &\cdots & -b_n\\
-b_1 &1-\lambda&0&\cdots &0\\
\vdots &0&1-\lambda& &\vdots\\
\vdots &\vdots&&\ddots &\vdots\\
-b_n &0&\cdots&\cdots &1-\lambda\\ \end{vmatrix}=(1-\lambda)^{n-1}\left((1-\lambda)(c-\lambda)-\sum_{i=1}^{n}b_i^2\right)
\\&=(1-\lambda)^{n-1}\left(\lambda^2 -(1+c)\lambda + c-\sum_{i=1}^{n}b_i^2\right)=0
\end{align*}
which implies
$$\lambda_1=1, \quad\lambda_{2,3}=\frac{(1+c) \pm \sqrt{(1-c)^2+4\sum_{i=1}^{n}b_i^2}}{2},$$
\begin{align}
\lambda_{\max}^2&=\frac{\left((1+c) + \sqrt{(1-c)^2+4\sum_{i=1}^{n}b_i^2}\right)^2}{4}
\leq \frac{2(1+c)^2 + 2(1-c)^2+8\sum_{i=1}^{n}b_i^2}{4} \leq\frac{4+4c^2+8}{4} \underbrace{\leq}_{\eqref{sumb}} 4 \implies \lambda_{\max}(G) \leq 2 \label{maxg}
\end{align} 
finally from \eqref{g} we obtain:
\begin{align}
\|\Lambda_{uv}^{\frac{1}{2}} \left(\begin{array}{c}
\frac{\tau}{2\sqrt{\gamma}}\bu_0 \\  \frac{\tau}{2}\bar{\bA}_{22}^{-\frac{1}{2}} \bv_0 \\\textbf{0} \\\textbf{0}
\end{array}\right)\|^2\leq \frac{\tau^2}{2} \left(\beta_s^{-2}\|\Lambda_3^{\frac{1}{2}}\bp\|^2+\beta_{v}^{-2}\|\Lambda_2^{-\frac{1}{2}}\bp\|^2\right) \label{uv}
\end{align}
Let now verify the boundedness of $(\bw;\bz;\bwd;\bzd;\bq)$ by $(\bu;\bv;\bud;\bvd;\bp)$.Firstly for  $(\bw;\bz;\bwd;\bzd)$ we have,
\begin{align*}
\|(\bw;\bz;\bwd;\bzd)\|^2_{\bUd\times\bU\times\bVd\times\bV} &= \|\left(\delta \bu - \frac{\tau}{2\sqrt{\gamma}}\bu_0;\delta \bv -  \frac{\tau}{2}\bar{\bA}_{22}^{-\frac{1}{2}} \bv_0;\delta\bud;\delta\bvd\right)\|^2_{\bUd\times\bU\times\bVd\times\bV}\nonumber\\
&= \frac{\tau^2\mu}{2}\|\eps(\delta \bu - \frac{\tau}{2\sqrt{\gamma}}\bu_0)\|^2 +\frac{\tau^2\lambda}{4}\|\divv (\delta \bu - \frac{\tau}{2\sqrt{\gamma}}\bu_0)\|^2 + \|\Lambda_{uv}^{\frac{1}{2}} \left(\begin{array}{c}
\delta \bu - \frac{\tau}{2\sqrt{\gamma}}\bu_0 \nonumber\\\delta \bv -  \frac{\tau}{2}\bar{\bA}_{22}^{-\frac{1}{2}} \bv_0\\\delta\bud\\ \delta\bvd
\end{array}\right)\|^2\\
& + \frac{\tau^2}{4}\|\Lambda^{-\frac{1}{2}}\left(\Divv (\delta \bv -  \frac{\tau}{2}\bar{\bA}_{22}^{-\frac{1}{2}} \bv_0)+\ba\Divu (\delta \bu - \frac{\tau}{2\sqrt{\gamma}}\bu_0)\right)\|^2,
\end{align*}
by applying triangle inequality it follows
\begin{align*}
\leq &\tau^2\mu\delta^2\|\eps(\bu)\|^2 + \tau^2\mu\|\eps(\frac{\tau}{2\sqrt{\gamma}}\bu_0)\|^2+\frac{\tau^2\delta^2\lambda}{2}\|\divv\bu\|^2 +\frac{\tau^2\lambda}{2}\|\divv \frac{\tau}{2\sqrt{\gamma}}\bu_0\|^2 + 2\delta^2\|\Lambda_{uv}^{\frac{1}{2}} \left(\begin{array}{c}
\bu\\ \bv \\\bud\\ \bvd
\end{array}\right)\|^2 
\\&+ 2\|\Lambda_{uv}^{\frac{1}{2}} \left(\begin{array}{c}
\frac{\tau}{2\sqrt{\gamma}}\bu_0 \\ \frac{\tau}{2}\bar{\bA}_{22}^{-\frac{1}{2}} \bv_0 \\\textbf{0} \\\textbf{0}
\end{array}\right)\|^2 + \frac{\tau^2\delta^2}{2}\|\Lambda^{-\frac{1}{2}}\left(\Divv \bv + \ba \Divu \bu \right)\|^2 + \frac{\tau^2}{2}\|\Lambda^{-\frac{1}{2}}( \frac{\tau}{2}\bar{\bA}_{22}^{-\frac{1}{2}}\Divv  \bv_0+ \frac{\tau}{2\sqrt{\gamma}}\ba \Divu \bu_0)\|^2, 
\end{align*}
by \eqref{psi} ,\eqref{u0} ,  \eqref{uv} and definition of $\g$, we have
\begin{align}
\leq& \tau^2\mu\delta^2\|\eps(\bu)\|^2 + \frac{\tau^2\beta_s^{-2}}{2}\|\Lambda_3^{\frac{1}{2}}\bp\|^2+\frac{\tau^2\delta^2\lambda}{2}\|\divv\bu\|^2 +\frac{\tau^2\beta_s^{-2}}{2}\|\Lambda_3^{\frac{1}{2}}\bp\|^2 + 2\delta^2\|\Lambda_{uv}^{\frac{1}{2}} \left(\begin{array}{c}
\bu\\ \bv\\\bud\\ \bvd
\end{array}\right)\|^2 \nonumber\\&
+ \tau^2 \left(\beta_s^{-2}\|\Lambda_3^{\frac{1}{2}}\bp\|^2+\beta_v^{-2}\|\Lambda_2^{\frac{1}{2}}\bp\|^2\right) + \frac{\tau^2\delta^2}{2}\|\Lambda^{-\frac{1}{2}}\left(\Divv \bv +\ba \Divu \bu \right)\|^2 + \frac{\tau^2}{2}\|\Lambda^{-\frac{1}{2}} (\Lambda_2 + \Lambda_3)\bp\|^2  \label{wzbound}
\end{align}
Secondly for $\bq$ we have
\begin{align}
\|-\delta \bp - \Lambda^{-1}(\Divv \bv+\ba\Divu \bu)\|^2_{\bP} &= \frac{\tau^2}{4}( \Lambda\left(-\delta \bp - \Lambda^{-1}(\Divv \bv+\ba\Divu \bu)\right), -\delta \bp - \Lambda^{-1}(\Divv \bv+\ba\Divu \bu))\nonumber\\
\intertext{by applying triangle inequality it follows}
&\leq \frac{\tau^2\delta^2}{2}\|\Lambda^{\frac{1}{2}}\bp\|^2 + \frac{\tau^2}{2}\|\Lambda^{-\frac{1}{2}}(\Divv \bv+\ba\Divu \bu)\|^2 \label{qbound}
\end{align}
Collecting the estimates \eqref{wzbound} and \eqref{qbound}, we obtain
$$\|(\bw;\bz;\bwd;\bzd)\|^2_{\bUd\times\bU\times\bVd\times\bV} +\|\bq\|^2_{\bP} \leq (2\delta^2 + 2 + 8\beta_s^{-2} + 4\beta_v^{-2})(\|(\bu;\bv;\bud;\bvd)\|^2_{\bUd\times\bU\times\bVd\times\bV} +\|\bp\|^2_{\bP})$$
Stays to show the coercivity of $\mathcal{A}((\bu; \bv;\bud;\bvd;\bp),(\bw; \bz;\bwd;\bzd;\bq))$. Using the definition of  $\mathcal{A}((\bu; \bv;\bud;\bvd;\bp),(\bw; \bz;\bwd;\bzd;\bq))$ and $(\bw; \bz;\bwd;\bzd;\bq)$ from~\eqref{wzq}, it follow
\begin{align*}
\mathcal{A}((\bu;\bv;\bud;\bvd;\bp)&,(\bw,\bz;\bwd;\bzd;\bq))=
\frac{\mu\tau^2}{2}(\eps(\bu),\eps(\bw)) +\frac{\lambda\tau^2}{4}(\divv \bu,\divv \bw)+ (\Lambda_{uv} \left(\begin{array}{c}
\bu \\\bv \\ \bud\\ \bvd
\end{array}\right),\left(\begin{array}{c}
\bw \\\bz\\ \bwd\\ \bzd
\end{array}\right)) \nonumber
\\&-\frac{\tau^2}{4}(\bp,\ba\Divu \bw+\Divv \bz) -\frac{\tau^2}{4}(\ba\Divu \bu+\Divv \bv,\bq)  - \frac{\tau^2}{4}(\Lambda_1 \bp,\bq)
\\&=\frac{\mu\tau^2}{2}(\eps(\bu),\eps(\delta \bu - \frac{\tau}{2\sqrt{\gamma}}\bu_0)) +\frac{\lambda\tau^2}{4}(\divv \bu,\divv (\delta \bu - \frac{\tau}{2\sqrt{\gamma}}\bu_0))
+ (\Lambda_{uv} \left(\begin{array}{c}
\bu \\\bv \\ \bud\\ \bvd
\end{array}\right),\left(\begin{array}{c}
\delta \bu - \frac{\tau}{2\sqrt{\gamma}}\bu_0 \\
\delta \bv -  \frac{\tau}{2}\bar{\bA}_{22}^{-\frac{1}{2}} \bv_0\\
\delta\bud\\
\delta\bvd
\end{array}\right))
\\&
-\frac{\tau^2}{4}(\bp,\ba\Divu (\delta \bu - \frac{\tau}{2\sqrt{\gamma}}\bu_0)
+\Divv (\delta \bv -  \frac{\tau}{2}\bar{\bA}_{22}^{-\frac{1}{2}} \bv_0))
-\frac{\tau^2}{4}(\ba\Divu \bu+\Divv \bv,-\delta \bp - \Lambda^{-1}(\Divv \bv+\ba\Divu \bu))
\\&
- \frac{\tau^2}{4}(\Lambda_1 \bp,-\delta \bp - \Lambda^{-1}(\Divv \bv+\ba\Divu \bu))
\end{align*}
from \eqref{psi} and \eqref{u0}, it follow,
\begin{align*} 
&=\frac{\delta\mu\tau^2}{2}\|\eps(\bu)\|^2 - 
\frac{\mu\tau^2}{2}(\eps(\bu),\eps(\frac{\tau}{2\sqrt{\gamma}}\bu_0)) +\frac{\delta\lambda\tau^2}{4}\|\divv \bu\|^2 - \frac{\lambda\tau^2}{4}(\divv \bu,\divv (\frac{\tau}{2\sqrt{\gamma}}\bu_0))
\\
&+ \delta\|\Lambda_{uv}^{\frac{1}{2}} \left(\begin{array}{c}
\bu \\\bv \\ \bud\\ \bvd
\end{array}\right)\|^2 - 
(\Lambda_{uv} \left(\begin{array}{c}
\bu \\\bv \\ \bud\\ \bvd
\end{array}\right),\left(\begin{array}{c}
\frac{\tau}{2\sqrt{\gamma}}\bu_0 \\
\frac{\tau}{2}\bar{\bA}_{22}^{-\frac{1}{2}} \bv_0\\
\textbf{0}\\
\textbf{0}
\end{array}\right)) +\frac{\tau^2}{4}(\bp,(\Lambda_2+\Lambda_3)\bp)
+\frac{\tau^2}{4}\|\Lambda^{-\frac{1}{2}}(\Divv \bv+\ba\Divu \bu)\|^2 + \frac{\delta\tau^2}{4}\|\Lambda_1^{\frac{1}{2}} \bp\|^2
\\&
+ \frac{\tau^2}{4}(\Lambda_1 \bp,\Lambda^{-1}(\Divv \bv+\ba\Divu \bu))
\end{align*}
by using Young's inequality, we obtain,
\begin{align*}
&\geq\frac{\delta\mu\tau^2}{2}\|\eps(\bu)\|^2 - 
\frac{\mu\tau^2\epsilon_1}{4}\|\eps(\bu)\|^2 - \frac{\mu\tau^2}{4\epsilon_1}\frac{\tau^2}{4\gamma}\|\eps(\bu_0)\|^2 +\frac{\delta\lambda\tau^2}{4}\|\divv \bu\|^2 - \frac{\lambda\tau^2\epsilon_2}{8}\|\divv \bu\|^2  -\frac{\lambda\tau^2}{8\epsilon_2}\frac{\tau^2}{4\gamma}\|\divv \bu_0\|^2 \nonumber\\
&+ \delta\|\Lambda_{uv}^{\frac{1}{2}} \left(\begin{array}{c}
\bu \\\bv \\ \bud\\ \bvd
\end{array}\right)\|^2 - \frac{\epsilon_3}{2}\|\Lambda_{uv}^{\frac{1}{2}} \left(\begin{array}{c}
\bu \\\bv \\ \bud\\ \bvd
\end{array}\right)\|^2- 
\frac{1}{2\epsilon_3}\|\Lambda_{uv}^{\frac{1}{2}}\left(\begin{array}{c}
\frac{\tau}{2\sqrt{\gamma}}\bu_0 \\
\frac{\tau}{2}\bar{\bA}_{22}^{-\frac{1}{2}} \bv_0\\
\textbf{0}\\
\textbf{0}
\end{array}\right)\|^2 +\frac{\tau^2}{4}\|\Lambda_2^{\frac{1}{2}}\bp\|^2 +\frac{\tau^2}{4}\|\Lambda_3^{\frac{1}{2}}\bp\|^2
\\& +\frac{\tau^2}{4}\|\Lambda^{-\frac{1}{2}}(\Divv \bv+\ba\Divu \bu)\|^2 + \frac{\delta\tau^2}{4}\|\Lambda_1^{\frac{1}{2}} \bp\|^2 - \frac{\tau^2}{8}\|\Lambda^{-\frac{1}{2}}\Lambda_1 \bp\|^2 -\frac{\tau^2}{8}\|\Lambda^{-\frac{1}{2}}(\Divv \bv+\ba\Divu \bu)\|^2
\end{align*}
by using again \eqref{u0} , \eqref{uv} and the definition of $\g$ and $\Lambda$, we obtain,
\begin{align*}
&\geq\frac{\delta\mu\tau^2}{2}\|\eps(\bu)\|^2 - 
\frac{\mu\tau^2\epsilon_1}{4}\|\eps(\bu)\|^2 - \frac{\beta_s^{-2}\tau^2}{8\epsilon_1}\|\Lambda_3^{\frac{1}{2}}\bp\|^2 +\frac{\delta\lambda\tau^2}{4}\|\divv \bu\|^2 - \frac{\lambda\tau^2\epsilon_2}{8}\|\divv \bu\|^2  -\frac{\beta_s^{-2}\tau^2}{8\epsilon_2}\|\Lambda_3^{\frac{1}{2}}\bp\|^2 \nonumber\\
&+ \delta\|\Lambda_{uv}^{\frac{1}{2}} \left(\begin{array}{c}
\bu \\\bv \\ \bud\\ \bvd
\end{array}\right)\|^2 - \frac{\epsilon_3}{2}\|\Lambda_{uv}^{\frac{1}{2}} \left(\begin{array}{c}
\bu \\\bv \\ \bud\\ \bvd
\end{array}\right)\|^2- 
\frac{1}{2\epsilon_3}\frac{\tau^2}{2} \left(\beta_s^{-2}\|\Lambda_3^{\frac{1}{2}}\bp\|^2+\beta_{v}^{-2}\|\Lambda_2^{-\frac{1}{2}}\bp\|^2\right)  +\frac{\tau^2}{4}\|\Lambda_2^{\frac{1}{2}}\bp\|^2 +\frac{\tau^2}{4}\|\Lambda_3^{\frac{1}{2}}\bp\|^2
\\& +\frac{\tau^2}{4}\|\Lambda^{-\frac{1}{2}}(\Divv \bv+\ba\Divu \bu)\|^2 + \frac{\delta\tau^2}{4}\|\Lambda_1^{\frac{1}{2}} \bp\|^2 - \frac{\tau^2}{8}\|\Lambda_1^{\frac{1}{2}} \bp\|^2 -\frac{\tau^2}{8}\|\Lambda^{-\frac{1}{2}}(\Divv \bv+\ba\Divu \bu)\|^2
\end{align*}

Let $\epsilon_1 = 2\beta_{s}^{-2} , \epsilon_2 = 2\beta_{s}^{-2}, \epsilon_3 = 4\max\lbrace{\beta_{v}^{-2},\beta_{s}^{-2}\rbrace}:= 4\beta^{-2}$, we obtain
\begin{align*}
&\geq\frac{(\delta-\beta_s^{-2})\mu\tau^2}{2}\|\eps(\bu)\|^2  - \frac{\tau^2}{16}\|\Lambda_3^{\frac{1}{2}}\bp\|^2 +\frac{(\delta-\beta_s^{-2})\lambda\tau^2}{4}\|\divv \bu\|^2 -\frac{\tau^2}{16}\|\Lambda_3^{\frac{1}{2}}\bp\|^2
+ (\delta-2\beta^{-2})\|\Lambda_{uv}^{\frac{1}{2}} \left(\begin{array}{c}
\bu \\\bv \\ \bud\\ \bvd
\end{array}\right)\|^2
\nonumber\\
& - 
\frac{\tau^2}{16} \left(\|\Lambda_3^{\frac{1}{2}}\bp\|^2+\|\Lambda_2^{-\frac{1}{2}}\bp\|^2\right)  +\frac{\tau^2}{4}\|\Lambda_2^{\frac{1}{2}}\bp\|^2 +\frac{\tau^2}{4}\|\Lambda_3^{\frac{1}{2}}\bp\|^2
+\frac{\tau^2}{4}\|\Lambda^{-\frac{1}{2}}(\Divv \bv+\ba\Divu \bu)\|^2 + \frac{\delta\tau^2}{4}\|\Lambda_1^{\frac{1}{2}} \bp\|^2 - \frac{\tau^2}{8}\|\Lambda_1^{\frac{1}{2}} \bp\|^2
\\&  -\frac{\tau^2}{8}\|\Lambda^{-\frac{1}{2}}(\Divv \bv+\ba\Divu \bu)\|^2
\end{align*}
Let $\delta:=2\beta^{-2}+ \frac{1}{4}$ we obtain Finally,
$$\mathcal{A}((\bu; \bv;\bud;\bvd;\bp),(\bw; \bz;\bwd;\bzd;\bq))\geq \frac{1}{4}(\|(\bu;\bv;\bud;\bvd)\|^2_{\bUd\times\bU\times\bVd\times\bV}+\|\bp\|^2_{\bP})$$
\end{proof}
The above theorem implies the following stability estimate.
\begin{corollary}\label{eq:67}
Let $(\bu;\bv;\bud;\bvd) \in \bUd\times\bU\times\bVd\times\bV\times \bP$ be the solution
of~\eqref{eq:MPET weak}. Then there holds the estimate
\begin{equation}
\|(\bu;\bv;\bud;\bvd)\|^2_{\bUd\times\bU\times\bVd\times\bV}+\|\bp\|_{\bP}\leq C_1 (\|(\bG_1;\bG_2;\bG_3;\bG_4)\|^2_{\bUd^*\times\bU^*\times\bVd^*\times\bV^*}+\|\bG_5\|_{\bP^*}),
\end{equation} 
where $C_1$ is a constant independent of all parameters and the network scale $n$ 
and $$\|(\bG_1;\bG_2;\bG_3;\bG_4)\|^2_{\bUd^*\times\bU^*\times\bVd^*\times\bV^*}=
\sup\limits_{(\bw;\bz;\bwd;\bzd)\in \bUd\times\bU\times\bVd\times\bV}\frac{((\bG_1;\bG_2;\bG_3;\bG_4), (\bw;\bz;\bwd;\bzd))}{\|(\bw;\bz;\bwd;\bzd)\|_{\bUd\times\bU\times\bVd\times\bV}}, \qquad \|\bG_5\|_{\bP^*}=
\sup\limits_{\bq\in \bP}\frac{(\bG_5,\bq)}{\|\bq\|_{\bP}}=\|\Lambda^{-\frac{1}{2}} \bG_5\|.$$
\end{corollary}
\subsection{Stability of the fully discrete problem}
The main result of this section is a proof of the uniform well-posedness of problem \eqref{a2dis} under the norms
induced by \eqref{U_hnorm} and \eqref{norm}. 
\begin{theorem}
There exists a constant $C_d$ independent of  all parameters , the network scale~$n$  and the mesh size $h$, such that  for any $(\bu_h;\bv_h;\bud_h;\bvd_h;\bp_h),(\bw_h,\bz_h;\bwd_h;\bzd_h;\bq_h)\in \bUd_h\times\bU_h\times\bVd_h\times\bV_h\times \bP_h$
\begin{align*}
|\mathcal{A}_h((\bu_h;\bv_h;\bud_h;\bvd_h;\bp_h),(\bw_h;\bz_h;\bwd_h;\bzd_h;\bq_h))|\leq C_d &(\|(\bu_h,\bv_h,\bud_h,\bvd_h)\|_{\bUd_h\times\bU\times\bVd\times\bV}+\|\bp_h\|_{\bP})\\\cdot&(\|(\bw_h,\bz_h,\bwd_h,\bzd_h)\|_{\bUd_h\times\bU\times\bVd\times\bV}+\|\bq_h\|_{\bP}).
\end{align*}\label{theorem-bound-dis} 
\end{theorem}
\begin{proof}
The proof of this theorem can be obtained by following the proof of Theorem \ref{theorem-bound} 
\end{proof}
The following theorem shows the inf-sup-condition (LBB) of $\mathcal{A}_h((.;.;.;.;.),(.;.;.;.;.))$
\begin{theorem}
There exists a constant $\omega_h> 0$ independent of all parameters , the network scale $n$ and  and the mesh size $h$, such that
\begin{align*}
\inf_{\substack {(\bu_h;\bv_h;\bud_h;\bvd_h;\bp_h)\in \\ \bUd_h\times \bU_h\times \bVd_h\times \bV_h\times\bP_h}}\sup_{\substack {(\bw_h;\bz_h;\bwd_h;\bzd_h;\bq_h)\in \\  \bUd_h\times \bU_h\times \bVd_h\times \bV_h\times\bP_h}}\frac{\mathcal{A}_h((\bu_h;\bv_h;\bud_h;\bvd_h;\bp_h),(\bw_h;\bz_h;\bwd_h;\bzd_h;\bq_h))}{(\|(\bu_h;\bv_h;\bud_h;\bvd_h)\|_{ \bUd_h\times \bU\times \bVd\times \bV}+\|\bp_h\|_{\bP})(\|(\bw_h;\bz_h;\bwd_h;\bzd_h)\|_{ \bUd_h\times \bU\times \bVd\times \bV}+\|\bq_h\|_{\bP})} \geq \omega_h.
\end{align*}\label{inf-sup-dis}
\end{theorem}
\begin{proof}
For any $(\bu_h;\bv_h;\bud_h;\bvd_h;\bp_h)\in \bUd_h\times\bU_h\times \bVd_h\times\bV_h\times \bP_h$, from \ref{divdis}, there exist
\begin{align}
\bv_{0,h} &\in \bU ~~\hbox{such that}~~\Divv\bv_{0,h} = \frac{\tau}{2}\bar{\bA}_{22}^{-\frac{1}{2}}\bp_h ~~\hbox{and}~~ \|\bv_{0,h}\|_{\divv} \leq { \beta_{v,h}^{-1}\|\Lambda_2^{-\frac{1}{2}}\bp_h\|} \label{psi-dis},
\end{align}
also, there exists
\begin{align}
\bu_{0,h} &\in \bUd_h~~\hbox{such that}~~ 
\Divu \bu_{0,h} = \frac{\tau}{2\sqrt{\gamma}} \Lambda_4\ba\bp_h ,~~\|\bu_{0,h}\|_1 \leq \frac{\tau\beta_{s}^{-1}}{2\sqrt{\gamma}}\|\Lambda_4^{\frac{1}{2}}\ba\bp_h\|= \beta_{s,h}^{-1}\|\Lambda_3^{\frac{1}{2}}\bp_h\|.\label{u0-dis}
\end{align}
Choose
\begin{align}
\bw_h = \delta \bu_h - \frac{\tau}{2\sqrt{\gamma}}\bu_{0,h},\quad \bz_h = \delta \bv_h -  \frac{\tau}{2}\bar{\bA}_{22}^{-\frac{1}{2}} \bv_{0,h}, ~~ \bwd_h=\delta\bud_h, ~~\bzd_h=\delta\bvd_h, ~~\bq_h = -\delta \bp_h - \frac{\tau^2}{4}\Lambda^{-1}(\Divv \bv_h+\ba\Divu \bu_h),\label{wzq-dis}
\end{align}
where $\delta$ is a positive constant to be determinant later.\\
Following the proof of Theorem \ref{inf-sup}, we try to estimate $\|\Lambda_{uv}^{\frac{1}{2}} \left(\begin{array}{c}
\frac{\tau}{2\sqrt{\gamma}}\bu_{0,h} \\  \frac{\tau}{2}\bar{\bA}_{22}^{-\frac{1}{2}} \bv_{0,h} \\\textbf{0} \\\textbf{0}
\end{array}\right)\|^2$:
	
\begin{align}
\|\Lambda_{uv}^{\frac{1}{2}} \left(\begin{array}{c}
\frac{\tau}{2\sqrt{\gamma}}\bu_{0,h} \\  \frac{\tau}{2}\bar{\bA}_{22}^{-\frac{1}{2}} \bv_{0,h} \\\textbf{0} \\\textbf{0}
\end{array}\right)&\|^2
\underbrace{\leq}_\eqref{g} \frac{\tau^2}{4}\lambda_{\max}(G)(\|\bu_{0,h}\|^2+\|\bv_{0,h}\|^2) \underbrace{\leq}_\eqref{Poincare} \frac{\tau^2}{4}\lambda_{\max}(G)\left(c_3\beta_{s,h}^{-2}\|\Lambda_3^{\frac{1}{2}}\bp_h\|^2+\beta_{v,h}^{-2}\|\Lambda_2^{-\frac{1}{2}}\bp_h\|^2\right)\nonumber
\\%
&\underbrace{\leq}_\eqref{maxg}\frac{\tau^2}{2}\left(c_3\beta_{s,h}^{-2}\|\Lambda_3^{\frac{1}{2}}\bp_h\|^2+\beta_{v,h}^{-2}\|\Lambda_2^{-\frac{1}{2}}\bp_h\|^2\right) \label{uv-dis}
\end{align}
Let now verify the boundedness of $(\bw_h;\bz_h;\bwd_h;\bzd_h;\bq_h)$ by $(\bu_h;\bv_h;\bud_h;\bvd_h;\bp_h)$. Firstly for  $(\bw_h;\bz_h;\bwd_h;\bzd_h)$ we have,
	
\begin{align*}
\|(\bw_h;\bz_h;\bwd_h;\bzd_h)\|^2_{\bUd_h\times\bU\times\bVd\times\bV} &= \|\left(\delta \bu_h - \frac{\tau}{2\sqrt{\gamma}}\bu_{0,h};\delta \bv_h -  \frac{\tau}{2}\bar{\bA}_{22}^{-\frac{1}{2}} \bv_{0,h};\delta\bud_h;\delta\bvd_h\right)\|^2_{\bUd_h\times\bU\times\bVd\times\bV}\nonumber\\
&= \frac{\tau^2\mu}{2}\|\delta \bu_h - \frac{\tau}{2\sqrt{\gamma}}\bu_{0,h}\|^2_{DG} +\frac{\tau^2\lambda}{4}\|\divv (\delta \bu_h - \frac{\tau}{2\sqrt{\gamma}}\bu_{0,h})\|^2 + \|\Lambda_{uv}^{\frac{1}{2}} \left(\begin{array}{c}
\delta \bu_h - \frac{\tau}{2\sqrt{\gamma}}\bu_{0,h} \nonumber\\\delta \bv_h -  \frac{\tau}{2}\bar{\bA}_{22}^{-\frac{1}{2}} \bv_{0,h}\\\delta\bud_h\\ \delta\bvd_h
\end{array}\right)\|^2\\
& + \frac{\tau^2}{4}\|\Lambda^{-\frac{1}{2}}\left(\Divv (\delta \bv_h -  \frac{\tau}{2}\bar{\bA}_{22}^{-\frac{1}{2}} \bv_{0,h})+\ba\Divu (\delta \bu_h - \frac{\tau}{2\sqrt{\gamma}}\bu_{0,h})\right)\|^2,
\end{align*}
by applying triangle inequality it follows
\begin{align*}
\leq &\tau^2\mu\delta^2\|\bu_h\|^2_{DG} + \tau^2\mu\|\frac{\tau}{2\sqrt{\gamma}}\bu_{0,h}\|^2_{DG}+\frac{\tau^2\delta^2\lambda}{2}\|\divv\bu_h\|^2 +\frac{\tau^2\lambda}{2}\|\divv \frac{\tau}{2\sqrt{\gamma}}\bu_{0,h}\|^2 + 2\delta^2\|\Lambda_{uv}^{\frac{1}{2}} \left(\begin{array}{c}
\bu_h\\ \bv_h \\\bud_h\\ \bvd_h
\end{array}\right)\|^2 
\\&+ 2\|\Lambda_{uv}^{\frac{1}{2}} \left(\begin{array}{c}
\frac{\tau}{2\sqrt{\gamma}}\bu_{0,h} \\ \frac{\tau}{2}\bar{\bA}_{22}^{-\frac{1}{2}} \bv_{0,h} \\\textbf{0} \\\textbf{0}
\end{array}\right)\|^2 + \frac{\tau^2\delta^2}{2}\|\Lambda^{-\frac{1}{2}}\left(\Divv \bv_h + \ba \Divu \bu_h \right)\|^2 + \frac{\tau^2}{2}\|\Lambda^{-\frac{1}{2}}( \frac{\tau}{2}\bar{\bA}_{22}^{-\frac{1}{2}}\Divv  \bv_{0,h}+ \frac{\tau}{2\sqrt{\gamma}}\ba \Divu \bu_{0,h})\|^2, 
\end{align*}
by  \eqref{psi-dis} ,\eqref{u0-dis} , \eqref{uv-dis} , \eqref{Korninequality} and definition of $\g,\Lambda$, we have
\begin{align}
\leq& \tau^2\mu\delta^2\|\bu_h\|^2_{DG} + \frac{c_0\tau^2\beta_{s,h}^{-2}}{2}\|\Lambda_3^{\frac{1}{2}}\bp_h\|^2 +\frac{\tau^2\delta^2\lambda}{2}\|\divv\bu_h\|^2 +\frac{\tau^2\beta_{s,h}^{-2}}{2}\|\Lambda_3^{\frac{1}{2}}\bp_h\|^2 + 2\delta^2\|\Lambda_{uv}^{\frac{1}{2}} \left(\begin{array}{c}
\bu_h\\ \bv_h\\\bud_h\\ \bvd_h
\end{array}\right)\|^2 \nonumber\\&
+ \tau^2 \left(c_3\beta_{s,h}^{-2}\|\Lambda_3^{\frac{1}{2}}\bp_h\|^2+\beta_v^{-2}\|\Lambda_2^{\frac{1}{2}}\bp_h\|^2\right) + \frac{\tau^2\delta^2}{2}\|\Lambda^{-\frac{1}{2}}\left(\Divv \bv_h +\ba \Divu \bu_h \right)\|^2 + \frac{\tau^2}{2}\|\Lambda^{-\frac{1}{2}}\left( (\Lambda_2 + \Lambda_3)\bp_h\right)\|^2  \label{wzbound-dis}
\end{align}
Secondly for $\bq_h$ we have
\begin{align}
\|-\delta \bp_h - \Lambda^{-1}(\Divv \bv_h+\ba\Divu \bu_h)\|^2_{\bP} &= \frac{\tau^2}{4}( \Lambda\left(-\delta \bp_h - \Lambda^{-1}(\Divv \bv_h+\ba\Divu \bu_h)\right), -\delta \bp_h - \Lambda^{-1}(\Divv \bv_h+\ba\Divu \bu_h))\nonumber\\
&\leq \frac{\tau^2\delta^2}{2}\|\Lambda^{\frac{1}{2}}\bp_h\|^2 + \frac{\tau^2}{2}\|\Lambda^{-\frac{1}{2}}(\Divv \bv_h+\ba\Divu \bu_h)\|^2 \label{qbound-dis}
\end{align}
Collecting the estimates \eqref{wzbound-dis} and \eqref{qbound-dis}, we obtain
$$\|(\bw_h;\bz_h;\bwd_h;\bzd_h)\|^2_{\bUd_h\times\bU\times\bVd\times\bV} +\|\bq_h\|^2_{\bP} \leq (2\delta^2 + 2 + (2+2c_0+4c_3)\beta_{s,h}^{-2} + 4\beta_v^{-2})(\|(\bu_h;\bv_h;\bud_h;\bvd_h)\|^2_{\bUd_h\times\bU\times\bVd\times\bV} +\|\bp_h\|^2_{\bP})$$

Stays to show the coercivity of $\mathcal{A}((\bu_h; \bv_h;\bud_h;\bvd_h;\bp_h),(\bw_h; \bz_h;\bwd_h;\bzd_h;\bq_h))$. Using the definition of

$\mathcal{A}((\bu_h; \bv_h;\bud_h;\bvd_h;\bp_h),(\bw_h; \bz_h;\bwd_h;\bzd_h;\bq_h))$ and $(\bw_h; \bz_h;\bwd_h;\bzd_h;\bq_h)$ from~\eqref{wzq-dis}, it follow
\begin{align*}
\mathcal{A}((&\bu_h;\bv_h;\bud_h;\bvd_h;\bp_h),(\bw_h,\bz_h;\bwd_h;\bzd_h;\bq_h))=
\frac{\mu\tau^2}{2}a_h(\bu_h,\bw_h) +\frac{\lambda\tau^2}{4}(\divv \bu_h,\divv \bw_h)+ (\Lambda_{uv} \left(\begin{array}{c}
\bu_h \\\bv_h \\ \bud_h\\ \bvd_h
\end{array}\right),\left(\begin{array}{c}
\bw_h \\\bz_h\\ \bwd_h\\ \bzd_h
\end{array}\right)) \nonumber
\\&-\frac{\tau^2}{4}(\bp_h,\ba\Divu \bw_h+\Divv \bz_h) -\frac{\tau^2}{4}(\ba\Divu \bu_h+\Divv \bv_h,\bq_h)  - \frac{\tau^2}{4}(\Lambda_1 \bp_h,\bq_h)
\\&=\frac{\mu\tau^2}{2}a_h(\bu_h,\delta \bu_h - \frac{\tau}{2\sqrt{\gamma}}\bu_{0,h}) +\frac{\lambda\tau^2}{4}(\divv \bu_h,\divv (\delta \bu_h - \frac{\tau}{2\sqrt{\gamma}}\bu_{0,h}))
+ (\Lambda_{uv} \left(\begin{array}{c}
\bu_h \\\bv_h \\ \bud_h\\ \bvd_h
\end{array}\right),\left(\begin{array}{c}
\delta \bu_h - \frac{\tau}{2\sqrt{\gamma}}\bu_{0,h} \\
\delta \bv_h -  \frac{\tau}{2}\bar{\bA}_{22}^{-\frac{1}{2}} \bv_{0,h}\\
\delta\bud_h\\
\delta\bvd_h
\end{array}\right))
\nonumber\\
& -\frac{\tau^2}{4}(\bp_h,\ba\Divu (\delta \bu_h - \frac{\tau}{2\sqrt{\gamma}}\bu_{0,h})+\Divv (\delta \bv_h -  \frac{\tau}{2}\bar{\bA}_{22}^{-\frac{1}{2}} \bv_{0,h}))
-\frac{\tau^2}{4}(\ba\Divu \bu_h+\Divv \bv_h,-\delta \bp_h - \Lambda^{-1}(\Divv \bv_h+\ba\Divu \bu_h))
\\&  - \frac{\tau^2}{4}(\Lambda_1 \bp_h,-\delta \bp_h - \Lambda^{-1}(\Divv \bv_h+\ba\Divu \bu_h))
\end{align*}
from \eqref{psi-dis} and \eqref{u0-dis}, it follow,
\begin{align*} 
&=\frac{\delta\mu\tau^2}{2}a_h(\bu_h,\bu_h) - \frac{\mu\tau^2}{2}a_h(\bu_h,\frac{\tau}{2\sqrt{\gamma}}\bu_{0,h}) +\frac{\delta\lambda\tau^2}{4}\|\divv \bu_h\|^2 - \frac{\lambda\tau^2}{4}(\divv \bu_h,\divv (\frac{\tau}{2\sqrt{\gamma}}\bu_{0,h}))
+ \delta\|\Lambda_{uv}^{\frac{1}{2}} \left(\begin{array}{c}
\bu_h \\\bv_h \\ \bud_h\\ \bvd_h
\end{array}\right)\|^2
\nonumber\\
& - 
(\Lambda_{uv} \left(\begin{array}{c}
\bu_h \\\bv_h \\ \bud_h\\ \bvd_h
\end{array}\right),\left(\begin{array}{c}
\frac{\tau}{2\sqrt{\gamma}}\bu_{0,h} \\
\frac{\tau}{2}\bar{\bA}_{22}^{-\frac{1}{2}} \bv_{0,h}\\
\textbf{0}\\
\textbf{0}
\end{array}\right)) +\frac{\tau^2}{4}(\bp_h,(\Lambda_2+\Lambda_3)\bp_h)
+\frac{\tau^2}{4}\|\Lambda^{-\frac{1}{2}}(\Divv \bv_h+\ba\Divu \bu_h)\|^2 + \frac{\delta\tau^2}{4}\|\Lambda_1^{\frac{1}{2}} \bp_h\|^2
\\&  + \frac{\tau^2}{4}(\Lambda_1 \bp_h,\Lambda^{-1}(\Divv \bv_h+\ba\Divu \bu_h))
\end{align*}
by using Young's inequality,\eqref{continuity:a_h} and \eqref{coercivity:a_h} we obtain,
\begin{align*}
&\geq\frac{\delta\alpha_a\mu\tau^2}{2}\|\bu_h\|^2_h - 
\frac{\mu\tau^2 c_2\epsilon_1}{4}\|\bu_h\|^2_{DG} - \frac{\mu\tau^2 c_2}{4\epsilon_1}\frac{\tau^2}{4\gamma}\|\bu_{0,h}\|^2_{DG} +\frac{\delta\lambda\tau^2}{4}\|\divv \bu_h\|^2 - \frac{\lambda\tau^2\epsilon_2}{8}\|\divv \bu_h\|^2  
-\frac{\lambda\tau^2}{8\epsilon_2}\frac{\tau^2}{4\gamma}\|\divv \bu_{0,h}\|^2 
\nonumber\\
&
+ \delta\|\Lambda_{uv}^{\frac{1}{2}} \left(\begin{array}{c}
\bu_h \\\bv_h \\ \bud_h\\ \bvd_h
\end{array}\right)\|^2 - \frac{\epsilon_3}{2}\|\Lambda_{uv}^{\frac{1}{2}} \left(\begin{array}{c}
\bu_h \\\bv_h \\ \bud_h\\ \bvd_h
\end{array}\right)\|^2- 
\frac{1}{2\epsilon_3}\|\Lambda_{uv}^{\frac{1}{2}}\left(\begin{array}{c}
\frac{\tau}{2\sqrt{\gamma}}\bu_{0,h} \\
\frac{\tau}{2}\bar{\bA}_{22}^{-\frac{1}{2}} \bv_{0,h}\\
\textbf{0}\\
\textbf{0}
\end{array}\right)\|^2 
+\frac{\tau^2}{4}\|\Lambda_2^{\frac{1}{2}}\bp_h\|^2  +\frac{\tau^2}{4}\|\Lambda_3^{\frac{1}{2}}\bp_h\|^2
\\&+\frac{\tau^2}{4}\|\Lambda^{-\frac{1}{2}}(\Divv \bv_h+\ba\Divu \bu_h)\|^2 + \frac{\delta\tau^2}{4}\|\Lambda_1^{\frac{1}{2}} \bp_h\|^2 - \frac{\tau^2}{8}\|\Lambda^{-\frac{1}{2}}\Lambda_1 \bp_h\|^2
-\frac{\tau^2}{8}\|\Lambda^{-\frac{1}{2}}(\Divv \bv_h+\ba\Divu \bu_h)\|^2
\end{align*}
by using again \eqref{u0-dis}, \eqref{uv-dis} , \eqref{Korninequality}, and the definition of $\g$, we obtain,
\begin{align*}
&\geq\frac{\alpha_a c_1 \delta\mu\tau^2}{2}\|\bu_h\|^2_{DG} - 
\frac{\mu\tau^2 c_0  c_2\epsilon_1}{4}\|\bu_h\|^2_{DG} - \frac{c_2\beta_{s,h}^{-2}\tau^2}{8\epsilon_1}\|\Lambda_3^{\frac{1}{2}}\bp_h\|^2 +\frac{\delta\lambda\tau^2}{4}\|\divv \bu_h\|^2 - \frac{\lambda\tau^2\epsilon_2}{8}\|\divv \bu_h\|^2
-\frac{\beta_{s,h}^{-2}\tau^2}{8\epsilon_2}\|\Lambda_3^{\frac{1}{2}}\bp_h\|^2 
\nonumber\\
&
+ \delta\|\Lambda_{uv}^{\frac{1}{2}} \left(\begin{array}{c}
\bu_h \\\bv_h \\ \bud_h\\ \bvd_h
\end{array}\right)\|^2 - \frac{\epsilon_3}{2}\|\Lambda_{uv}^{\frac{1}{2}} \left(\begin{array}{c}
\bu_h \\\bv_h \\ \bud_h\\ \bvd_h
\end{array}\right)\|^2- 
\frac{1}{2\epsilon_3}\frac{\tau^2}{2} \left(c_3\beta_{s,h}^{-2}\|\Lambda_3^{\frac{1}{2}}\bp_h\|^2+\beta_{v,h}^{-2}\|\Lambda_2^{-\frac{1}{2}}\bp_h\|^2\right)
+\frac{\tau^2}{4}\|\Lambda_2^{\frac{1}{2}}\bp_h\|^2 +\frac{\tau^2}{4}\|\Lambda_3^{\frac{1}{2}}\bp_h\|^2 
\\&
+\frac{\tau^2}{4}\|\Lambda^{-\frac{1}{2}}(\Divv \bv_h+\ba\Divu \bu_h)\|^2 + \frac{\delta\tau^2}{4}\|\Lambda_1^{\frac{1}{2}} \bp_h\|^2 - \frac{\tau^2}{8}\|\Lambda_1^{\frac{1}{2}} \bp_h\|^2 
-\frac{\tau^2}{8}\|\Lambda^{-\frac{1}{2}}(\Divv \bv_h+\ba\Divu \bu_h)\|^2
\end{align*}
Let $\epsilon_1 = 2\beta_{s,h}^{-2}c_2c_0 , \epsilon_2 = 2\beta_{s,h}^{-2}, \epsilon_3 = 4\max\lbrace{\beta_{v,h}^{-2},c_3\beta_{s,h}^{-2}\rbrace}:= 4\beta^{-2}_h$, we obtain
\begin{align*}
&\geq\frac{(\delta\alpha_a c_1-\beta_{s,h}^{-2}c_2^2 c_0)\mu\tau^2}{2}\|\bu_h\|^2_{DG}  - \frac{\tau^2}{16}\|\Lambda_3^{\frac{1}{2}}\bp_h\|^2 +\frac{(\delta-\beta_{s,h}^{-2})\lambda\tau^2}{4}\|\divv \bu_h\|^2 -\frac{\tau^2}{16}\|\Lambda_3^{\frac{1}{2}}\bp_h\|^2 + (\delta-2\beta^{-2}_h)\|\Lambda_{uv}^{\frac{1}{2}} \left(\begin{array}{c}
\bu_h \\\bv_h \\ \bud_h\\ \bvd_h
\end{array}\right)\|^2 \nonumber\\
& - 
\frac{\tau^2}{16} \left(\|\Lambda_3^{\frac{1}{2}}\bp_h\|^2+\|\Lambda_2^{-\frac{1}{2}}\bp_h\|^2\right)  +\frac{\tau^2}{4}\|\Lambda_2^{\frac{1}{2}}\bp_h\|^2 +\frac{\tau^2}{4}\|\Lambda_3^{\frac{1}{2}}\bp_h\|^2
 +\frac{\tau^2}{4}\|\Lambda^{-\frac{1}{2}}(\Divv \bv_h+\ba\Divu \bu_h)\|^2 + \frac{\delta\tau^2}{4}\|\Lambda_1^{\frac{1}{2}} \bp_h\|^2 - \frac{\tau^2}{8}\|\Lambda_1^{\frac{1}{2}} \bp_h\|^2
\\& 
-\frac{\tau^2}{8}\|\Lambda^{-\frac{1}{2}}(\Divv \bv_h+\ba\Divu \bu_h)\|^2
\end{align*}
Let $\delta:=\frac{\max\lbrace\beta_{s,h}^{-2}c_2^2 c_0 ,2\beta^{-2}_h,\beta_{s,h}^{-2}\rbrace}{\min\lbrace\alpha_a c_1,1\rbrace} + \frac{1}{4}$ we obtain Finally,
$$\mathcal{A}((\bu_h; \bv_h;\bud_h;\bvd_h;\bp_h),(\bw_h; \bz_h;\bwd_h;\bzd_h;\bq_h))\geq \frac{1}{4}(\|(\bu_h;\bv_h;\bud_h;\bvd_h)\|^2_{\bUd_h\times\bU\times\bVd\times\bV}+\|\bp_h\|^2_{\bP})$$
\end{proof}
The following stability estimate is a consequence of the above theorem.
\begin{corollary}\label{eq:92}
Let $(\bu_h;  \bv_h;\bud_h;\bvd_h; \bp_h)\in \bUd_h\times \bU_h\times \bVd_h\times \bV_h\times\bP_h$ be the solution of \eqref{a2dis},
then we have the estimate
\begin{equation}
\|(\bu_h;\bv_h;\bud_h;\bvd_h)\|_{\bUd_h\times \bU\times \bVd\times \bV} + \|\bp_h\|_{\bP}\leq C_2 (\|(\bG_1;\bG_2;\bG_3;\bG_4)\|_{\bUd_h^*\times \bU^*\times \bVd^*\times \bV^*}+\|\bG_5\|_{\bP^*}),\label{stability:dis}
\end{equation} 
holds where $$\|(\bG_1;\bG_2;\bG_3;\bG_4)\|_{\bUd_h^*\times \bU^*\times \bVd^*\times \bV^*}=
\sup\limits_{(\bw;\bz;\bwd;\bzd)\in \bUd_h\times \bU_h\times \bVd_h\times \bV_h}\frac{((\bG_1;\bG_2;\bG_3;\bG_4), (\bw;\bz;\bwd;\bzd))}{\|(\bw;\bz;\bwd;\bzd)\|_{\bUd_h\times \bU\times \bVd\times \bV}},\quad \|\bG_5\|_{\bP^*}=
\sup\limits_{\bq\in \bP_h}\frac{(\bG_5,\bq)}{\|\bq\|_{\bP}}=\|\Lambda^{-\frac{1}{2}} \bG_5\|.$$
and $C_2$ is a constant independent of all parameters the network scale  $n$, and the mesh size $h$.
\end{corollary}

\subsection{Consequences}
\subsubsection{Norm-equivalent preconditioner}

\begin{theorem2}
Let $\Lambda=(\tilde{\gamma}_{ij})_{n\times n}, \Lambda^{-1}=(\bar{\gamma}_{ij})_{n\times n}$.
Define 
\begin{equation}
\mathcal{B}:=\begin{bmatrix}
\mathcal{B}^{-1}_{\bu\bv} & \bm 0\\
\bm 0&\mathcal{B}^{-1}_{\bp}
\end{bmatrix}\label{Preconditioner}
\end{equation}
where 
\begin{align*}
\mathcal{B}_{\bu\bv} = \frac{\tau^2}{4}\begin{bmatrix}
\mathcal{\tilde{B}}_{\bu\bv}&\bm 0\\
\bm 0&\bm 0
\end{bmatrix}  +\Lambda_{uv},\qquad
\mathcal{B}_{\bp}&=
\frac{\tau^2}{4}\begin{bmatrix}
\tilde{\gamma}_{11}I&\tilde{\gamma}_{12}I& \dots &\tilde{\gamma}_{1n}I\\
\tilde{\gamma}_{21}I&\tilde{\gamma}_{22}I& \dots &\tilde{\gamma}_{2n}I\\
\vdots&\vdots  &\ddots&\vdots\\
\tilde{\gamma}_{n1}I&\tilde{\gamma}_{n2}I& \dots &\tilde{\gamma}_{nn}I\\
\end{bmatrix}, 
\end{align*}
\begin{align*}
\mathcal{\tilde{B}}_{\bu\bv} :=&
\begin{bmatrix}
- 2\mu\divv\eps - \lambda\nabla\divv - \sum_{i,j=1}^{n}\alpha_i\bar{\gamma}_{ij}\alpha_j\nabla\divv&&-\sum_{i=1}^{n}\alpha_i\bar{\gamma}_{i1}\nabla\divv & \cdots & -\sum_{i=1}^{n}\alpha_i\bar{\gamma}_{in}\nabla\divv \\\\
-\sum_{i=1}^{n}\alpha_i\bar{\gamma}_{i1}\nabla\divv && -\bar{\gamma}_{11}\nabla\divv & \cdots& -\bar{\gamma}_{1n}\nabla\divv \ \\
\vdots &&\vdots & \ddots & \vdots \\
-\sum_{i=1}^{n}\alpha_i\bar{\gamma}_{in}\nabla\divv &&-\bar{\gamma}_{n1}\nabla\divv & \cdots &  -\bar{\gamma}_{nn}\nabla\divv\\
\end{bmatrix} 
\end{align*}
	
Emulating from the theory presented in~\cite{Mardal2011preconditioning}, 
Theorems~\ref{theorem-bound} and~\ref{inf-sup} imply that the operator~$\mathcal{B}$ in~\eqref{Preconditioner}	defines a norm-equivalent (canonical) block-diagonal preconditioner for the operator $\mathcal{A}$ 
which is robust in all model parameters.
\end{theorem2}

\begin{theorem2}
Let $\textbf{W}_h :=\bUd_h\times\bU_h\times\bVd_h\times\bV_h\times\bP_h$ be equipped with the norm $\|\cdot\|^2_{\bm W_h}:=\|\cdot\|^2_{\bUd_h\times\bU\times\bVd\times\bV} + \|\cdot\|^2_{\bP}$ and consider the operator 
{
\begin{align*}
\mathcal{A}_h:=\left[ \begin{array}{ccccccccccccccc}
-\divv_h\sig_h + \g_u & -\g_1 &  \cdots & -\g_n & &-2\tau^{-1}& &0&\cdots&0& \alpha_1\nabla_h & \cdots  & \alpha_n\nabla_h\\\\
-\g_1 & \g_{v,1}& \cdots &0 & &0& &-2\tau^{-1}&\cdots&0&\nabla_h  & \cdots & 0\\
\vdots&  &\ddots &\vdots & & \vdots& & \vdots&\ddots& \vdots& \vdots & \ddots & \vdots\\
-\g_n & 0& \cdots &\g_{v,n} & &0& &0&\cdots&-2\tau^{-1}& 0 & \cdots  & \nabla_h\\\\
-2\tau^{-1} & 0    &\cdots&0     & &1      & &0     &\cdots&0     & 0    & \cdots & 0\\\\
0     &-2\tau^{-1}&\cdots&0          & &0      & &1     &\cdots&0     & 0    & \cdots & 0\\
\vdots&\vdots     &\ddots&\vdots     & &\vdots & &\vdots&\ddots&\vdots&\vdots& \ddots & \vdots\\
0     & 0         &\cdots&-2\tau^{-1}& &0      & &0     &\cdots&1     &0     & \cdots & 0\\\\
-\alpha_1\divv_h & -\divv_h& \cdots &0      & &0     & &0&\cdots&0& -\beta_{11}  & \cdots & \beta_{1n}\\
\vdots        & \vdots&\ddots  &\vdots & &\vdots& &\vdots & \ddots & \vdots& \vdots & \ddots & \vdots\\
-\alpha_n\divv_h & 0     & \cdots &-\divv_h & &0     & &0&\cdots&0& \beta_{n1}  & \cdots & -\beta_{nn}\\
\end{array}\right] 
\end{align*}
}
induced by the bilinear form \eqref{dis-Bilinear-form}. Clearly, $\mathcal{A}_h$ is self-adjoint and indefinite on $\textbf{W}_h$. Moreover, Theorems \ref{theorem-bound-dis} and \ref{inf-sup-dis}
imply that it is a uniform isomorphism in the sense of being bounded and having a bounded inverse with bounds independent of the mesh size, the network scale, and the model parameters. Following the framework in the study of Mardal et al.\cite{Mardal2011preconditioning}, we define the self-adjoint positive definite operator 
\begin{equation}\label{Preconditioner:B}
\mathcal{B}_h:=\left[\begin{array}{ccc}
\mathcal{B}_{h,\bm{uv}}^{-1}  & \bm 0 \\
	
\bm 0& \mathcal{B}_{h,\bm p}^{-1}  
\end{array}
\right],
\end{equation}
where 
\begin{align*}
\mathcal{B}_{h,\bu\bv} = \frac{\tau^2}{4}\begin{bmatrix}
\mathcal{\tilde{B}}_{h,\bu\bv}&\bm 0\\
\bm 0&\bm 0
\end{bmatrix}  +\Lambda_{uv},\qquad
\mathcal{B}_{\bp_h}&=\frac{\tau^2}{4}
\begin{bmatrix}
\tilde{\gamma}_{11}I&\tilde{\gamma}_{12}I& \dots &\tilde{\gamma}_{1n}I\\
\tilde{\gamma}_{21}I&\tilde{\gamma}_{22}I& \dots &\tilde{\gamma}_{2n}I\\
\vdots&\vdots  &\ddots&\vdots\\
\tilde{\gamma}_{n1}I&\tilde{\gamma}_{n2}I& \dots &\tilde{\gamma}_{nn}I\\
\end{bmatrix}, 
\end{align*}
{
\begin{align*}
\mathcal{\tilde{B}}_{h,\bu\bv} :=&
\begin{bmatrix}
- 2\mu\divv_h\eps - \lambda\nabla_h \divv_h - \sum_{i,j=1}^{n}\alpha_i\bar{\gamma}_{ij}\alpha_j\nabla_h \divv_h&&-\sum_{i=1}^{n}\alpha_i\bar{\gamma}_{i1}\nabla_h \divv_h & \cdots & -\sum_{i=1}^{n}\alpha_i\bar{\gamma}_{in}\nabla_h \divv_h \\\\
-\sum_{i=1}^{n}\alpha_i\bar{\gamma}_{i1}\nabla_h \divv_h && -\bar{\gamma}_{11}\nabla_h \divv_h & \cdots& -\bar{\gamma}_{1n}\nabla_h \divv_h \ \\
\vdots &&\vdots & \ddots & \vdots \\
-\sum_{i=1}^{n}\alpha_i\bar{\gamma}_{in}\nabla_h \divv_h &&-\bar{\gamma}_{n1}\nabla_h \divv_h & \cdots &  -\bar{\gamma}_{nn}\nabla_h \divv_h\\
\end{bmatrix}
\end{align*}
}
It is obvious that
$$\langle \mathcal{B}_h^{-1} \bx_h,\bx_h\rangle \eqsim \|\bx_h\|^2_{\bm W_h},$$
where $\bx_h = (\bu_h, \bv_h,\bud_h, \bvd_h, \bp_h) \in \bm W_h$ "$\eqsim$" stands for a norm equivalence, uniform with respect to model and discretization
parameters; and $\langle\cdot,\cdot\rangle$ expresses the duality pairing between $\bm W_h$ and $\bm W_h^*$, that is,$\mathcal{B}_h^{-1}$ is a uniform isomorphism.
By using the properties of $\mathcal{B}_h$ and $\mathcal{A}_h$ when solving the generalized eigenvalue problem $\mathcal{A}_h\bx_h =\xi \mathcal{B}_h^{-1}\bx_h$, the condition number $\kappa(\mathcal{B}_h\mathcal{A}_h)$ is easily shown to be uniformly bounded with respect to the all parameter, the network scale $n$, and the mesh size $h$. Therefore, $\mathcal{B}_h$ defines a uniform preconditioner.
\end{theorem2}

\section{Conclusions}\label{conclusion}
In this paper, we analyze the stability properties of the time-discrete systems arising from second-order implicit time stepping
schemes applied to the variational formulation of the MPET model and prove an inf-sup condition with a constant
that is independent of all model parameters. Moreover, we show that the fully discrete models obtained for a family
of strongly conservative space discretizations are also uniformly stable with respect to the spatial discretization
parameter. The norms in which these results hold are the basis for parameter-robust preconditioners
The transfer of the canonical (norm-equivalent) operator
preconditioners from the continuous and the discrete level lays the foundation for optimal and fully robust iterative solution methods.

\section*{Acknowledgement}
I would like to show my deep appreciation to my PHD supervisor Johannes Kraus and Maria Lymbery for thier guidance and help.
\bibliographystyle{elsarticle-num} 
\bibliography{mpet_dynamic}
\end{document}